\documentclass[12pt]{article}
\usepackage{amssymb,amsmath,amsthm}
\usepackage{indentfirst}
\usepackage[pdftex]{color,graphicx}
\usepackage[pdftex,bookmarks,unicode,colorlinks]{hyperref}

\def\R{\mathbb{R}}

\def\kasten{$~~\mbox{\hfil\vrule height6pt width5pt depth-1pt}$ }

\theoremstyle{plain}
\newtheorem{Theorem}{Theorem}[section]
\newtheorem{Corollary}[Theorem]{Corollary}

\newtheorem{Proposition}[Theorem]{Proposition}
\newtheorem{Lemma}[Theorem]{Lemma}
\newtheorem{Example}[Theorem]{Example}
\newtheorem{Remark}[Theorem]{Remark}
\newtheorem{Assumption}{Assumption}

\begin{document}

\def\kasten{$~~\mbox{\hfil\vrule height6pt width5pt depth-1pt}$ }
\centerline{\large \bf Maximum principles for nonlocal }
\centerline{\large \bf     parabolic Waldenfels operators }

\bigskip
\centerline{\bf
Qiao Huang$^{a,}\footnote{hq932309@hust.edu.cn}$,
Jinqiao Duan$^{a, b,}\footnote{duan@iit.edu}$ and
Jiang-Lun Wu$^{c,}\footnote{j.l.wu@swansea.ac.uk}$}
\smallskip
\centerline{${}^a$ Center for Mathematical Sciences }
\centerline{Huazhong University of Science and Technology}
\centerline{Wuhan, Hubei   430074, China}

\smallskip
\centerline{${}^b$ Department of Applied Mathematics, Illinois
Institute of Technology} \centerline{Chicago, IL 60616, USA}

\smallskip
\centerline{${}^c$ Department of Mathematics, Swansea University }
\centerline{Singleton Park, Swansea SA2 8PP, UK}

\smallskip
\centerline{April 17, 2018}

\bigskip\par
\begin{abstract}
As a class of L\'evy type Markov generators, nonlocal Waldenfels operators appear naturally in the context of investigating stochastic dynamics under L\'evy fluctuations and constructing Markov processes with boundary conditions (in particular the construction with jumps). This work is devoted to prove the weak and strong maximum principles for `parabolic' equations with nonlocal Waldenfels operators. Applications in stochastic differential equations with $\alpha$-stable L\'evy processes are presented to illustrate the maximum principles.
\end{abstract}

{\small \medskip\par\noindent {\bf Mathematics Subject
Classification (2010)}: 35B50; 35R09; 47G20; 60J75.

}

\section{Introduction}

The usual maximum principle concerns with second-order differential operators
of elliptic or parabolic type. It is a basic property of
solutions to boundary value problems for the associated elliptic or
parabolic partial differential equations (PDEs)  in a bounded domain. See
\cite{ProtterWeinberger,Renardy}) for a general study of maximum principles.
Classically, the maximum principle states that the maximum of the
  solution of a second-order elliptic or
parabolic equation in a domain is to be found on the boundary of
that domain. In particular, the strong maximum principle says that
if the solution achieves its maximum in the interior of the domain,
the solution must be   a constant, while the weak maximum
principle indicates that the maximum   is to be found
on the boundary but may re-occur in the interior as well. Let us
also mention \cite{Kuwae} where both weak and strong maximum principle
for symmetric Markov generators are discussed via (local) Dirichlet forms.
Moreover,  a maximum principle   for nonlocal operators generated
by nonnegative kernels defined on topological groups acting continuously
on a Hausdorff space was considered by Coville   \cite{Cov}.
The strong maximum principle for semicontinuous viscosity solution of
fully nonlinear second-order parabolic integro-differential equations was studied in \cite{Cio}.

A fairly large class of Markov processes on
$\mathbb{R}^d$ are governed analytically by their infinitesimal
generators, called L\'evy type generators or pseudo-differential
operators associated with negative definite symbols (cf. e.g.
\cite{Jacob}), either via martingale problem (cf. e.g.
\cite{Komatsu73,Komatsu00,Str,Stroock,Kolokoltsov}) or via Dirichlet
form (cf. e.g. \cite{Komatsu0,Komatsu,FOT,Jacob}). From
\cite{Jacob,Cou65}, these operators are usually integro-differential
operators or nonlocal operators, consisting of a combination of second-order
elliptic differential operators and integral operators of L\'evy
type. The nonlocal operator here corresponds to the jump component of a Markov process; in fact, it is an integral with respect to a jump measure.

The well-known Hille-Yosida theorem and the semigroup
approach, which can be found in e.g. \cite{kallenberg}, provide an intrinsic link between Markov processes and partial
differential equations, in particular second-order elliptic
differential operators,  as in the pioneering work of Feller in early 1950s. The monograph \cite{KTaira} (also
references therein) explores the functional analytic approach to
constructing Markov processes in a prescribed region of
$\mathbb{R}^d$, via the elliptic boundary value problems for the
associated L\'evy-type generators.

Due to the nature of  pseudo-differential operators (involving integral operators), the
L\'evy-type generators are nonlocal operators.
This kind of integro-differential operators was initiated by Waldenfels
\cite{Waldenfels} in 1960s. It was elucidated in \cite{KTaira} that a Markov
process associated with such an operator as infinitesimal operator  could be interpreted with a physical
picture: A Markovian particle moves both by jumps and
continuously in a certain region of the state space $\mathbb{R}^d$.

The present paper is devoted to the weak and strong maximum principles for the following nonlocal  parabolic  Waldenfels operator $-\frac{\partial}{\partial t} + L  $:
\begin{equation*}
  \begin{split}
     &\Big(-\frac{\partial}{\partial t} + L\Big) u(x,t)\\
      :=&  -\frac{\partial u}{\partial t}(x,t) +  \sum^d_{j,k=1}a_{jk}(x,t)\frac{\partial^2 u}{\partial x_j\partial x_k}(x,t)+\sum^d_{j=1}b_{j}(x,t)\frac{\partial u}{\partial x_j}(x,t)+c(x,t)u(x,t) \\
       & +\int_{\R^d\setminus\{0\}}\Big[u(x+z,t)-u(x,t)-\sum^d_{j=1}z_j \frac{\partial u}{\partial x_j}(x,t)\mathbf{1}_{\{|z|<1\}}\Big]\nu(t,x,dz),
  \end{split}
\end{equation*}
where the kernel $\{\nu(t,x,\cdot)\mid(x,t)\in\R^d\times[0,\infty)\}$   behaves as the jump measure for the associated Markov process.   The operator $L$ is called an elliptic  Waldenfels operator. Note that Waldenfels operators $L$ and
  $-\frac{\partial}{\partial t} + L  $  appear in the generator and  in the
Fokker-Planck equation, respectively, for  a stochastic differential equation  with  L\'evy motions \cite{Duan01, App, Dua15, Shlesinger95}.
We would like to point out
that Waldenfels operators also appear in nonlocal
  conservation laws \cite{TWu06}.   Certain properties for diffusion generators perturbed by the nonlocal
   Laplacian operator have also been studied   recently \cite{ARWu,ARWu1}.

We will prove the  new weak and strong maximum principles for the  nonlocal parabolic operator
 $-\frac{\partial}{\partial t} + L  $,   and they do not require any ``nondegeneracy'' conditions.
In order to cover the general case with either bounded or unbounded support of the jump measure $\nu$, we will introduce two open sets $D$ and $E$ (with $D \subset E$), where $D$ is the set where the maximum is achieved, and the stochastic process (``Markovian particle'') cannot jump from $D$ to the complement of $E$.

As a preparation for proving these maximum principles, we will     prove the maximum principles for nonlocal elliptic Waldenfels operator $L$. These  maximum principles   are important for the
 construction of Markov processes. In
\cite[Appendix C]{KTaira}, weak and strong maximum principles   for such elliptic Waldenfels
operators were proven,   but under stringent conditions, that is,  the jump measure  has to have bounded support.
The results in \cite{Cio} includes a strong maximum principle for viscosity solutions of certain \emph{nonlinear} nonlocal partial differential equations   under  a ``nondegeneracy'' condition.

The rest of this paper is organised as follows. In Section \ref{mpe},
we will present our results on maximum principles for elliptic Waldenfels operators. As a corollary, we also obtain the Hopf's Lemma about the sign of the gradient on the boundary.
Section \ref{mpp} is devoted to prove the maximum principles for parabolic Waldenfels operators.
Some consequences and examples are presented in Section \ref{exp}. Finally in Section \ref{appendix}, we present the proofs of some technical lemmas for the sake of completeness.

\section{Maximum principles for elliptic Waldenfels operators}\label{mpe}

In this section, we consider the weak and strong maximum principles for the elliptic Waldenfels operator $L$  (decomposed into local and nonlocal components)
\begin{equation}\label{ewo}
  L := A+K,
\end{equation}
where $A$ and $K$ are defined   as
\begin{equation*}
  \begin{split}
     Au(x) :=& \ \sum^d_{j,k=1}a_{jk}(x)\frac{\partial^2 u}{\partial x_j\partial x_k}(x)+\sum^d_{j=1}b_{j}(x)\frac{\partial u}{\partial x_j}(x)+c(x)u(x), \\
     Ku(x) :=& \int_{\R^d\setminus\{0\}}\Big[u(x+z)-u(x)-\sum^d_{j=1}z_j \frac{\partial u}{\partial x_j}(x)\mathbf{1}_{\{|z|<1\}}\Big]\nu(x,dz).
  \end{split}
\end{equation*}
Note that the coefficients are taken to be independent of time $t$. Note that the operator $K$ is actually the nonlocal Laplacian operator $-(-\Delta )^{\frac{\alpha}2}$, when the jump measure $\nu$ is the $\alpha$-stable type; see \cite[Ch. 7]{Dua15}.

The  elliptic Waldenfels operator $L$
plays an important role \cite{KTaira} in the theory of
Markov processes constructed in a given domain of $\mathbb{R}^d$. In that context, the second-order differential operator describes
the diffusion part of the associated Markov process and the integral
operator of L\'evy type corresponds to the jump behavior of the
Markov process. Finally, there is an assumption in that context which indicates that a
Markovian particle cannot move by jumps from any interior point of certain domain to the
outside of closure of the domain.  For further remarks and discussions, we refer
e.g. to Bony, Courr\`ege and Priouret \cite{BCP68} and Taira \cite{KTaira}.

\vspace{3mm}

To cover  more general situations,   we introduce two open sets $D$ and $E$ in  $\R^d$, with    $D\subset E$ and $E$   not necessarily bounded.  As usual, we denote the boundary of $D$ by $\partial D$, its closure by $\overline{D}:=D\cup\partial D$ and its complement by $D^c:=\R^d\setminus D$.

We make following assumptions:
\begin{enumerate}
  \item Continuity condition: $a_{jk},b_j,c\in C(\overline{E})$ $(j,k=1,...,d).$
  \item Symmetry condition: $a_{jk}=a_{kj}$ $(j,k=1,...,d)$. \\
  Uniform ellipticity condition:
      there exists a constant $\gamma>0$ such that
      \begin{equation}\label{pd}
        \sum^d_{j,k=1}a_{jk}(x)\xi_j\xi_k\ge\gamma|\xi|^2,
      \end{equation}
      for all $x\in D$, $\xi\in\R^d$.
  \item L\'evy measures: The kernel $\{\nu(x,\cdot)\mid x\in\R^d\}$ is a family of L\'evy measures, namely, each $\nu(x,\cdot)$ is a Borel measure on $\mathbb{R}^d\setminus\{0\}$ such that
      \begin{equation}\label{LM}
        \sup_{x\in\R^d} \int_{\mathbb{R}^d\setminus\{0\}}(1\land|z|^2)\nu(x,dz)<\infty,
      \end{equation}
      and moreover, for fixed $U\in\mathcal{B}(\mathbb{R}^d\setminus\{0\})$, the mapping $\R^d\ni x\to\nu(x,U)\in[0,\infty)$ is Borel measurable. Here we further assume that for each $x\in D$ the measure $\nu(x,\cdot)$ is supported in $\overline E-x:=\{y-x\mid y\in\overline E\}=\{z\mid x+z\in\overline E\}$, i.e.,
      \begin{equation}\label{supportNu2}
        \text{\upshape supp}\,\nu(x,\cdot)\subset\overline E-x,\quad \forall x\in D.
      \end{equation}
\end{enumerate}

\begin{Remark}\label{special}
  The support condition (\ref{supportNu2})   means in probability sense that a Markovian particle cannot move by jumps from a point $x\in D$ to the outside of $\overline{E}$. The motivation for this condition is that the maximizer point will propagate between connected components of the set in which the subsolution achieves maximum. The details will be discussed again in Remark \ref{tran} below. When   the set $E$ is the whole space $\R^d$, $E-x$ is still the whole space, and then there are actually no extra restrictions on the support of each measure $\nu(x,\cdot)$. In the case that $E=D$, the support condition is $\text{\upshape supp}\,\nu(x,\cdot)\subset \overline D-x$, and this is related to the assumption   in \cite{KTaira} that a Markovian particle cannot move by jumps from a point $x\in D$ to the outside of $\overline{D}$.
\end{Remark}

For convenience,  the notation $\mathbf{a}=(a_{jk})_{j,k=1,...,d}$   means $\mathbf{a}$ is a matrix with $(j,k)$-th entry $a_{jk}$, and $b=(b_1,...,b_d)^T$ is regarded as a row vector. We also recall the gradient operator (for space variable) $\nabla_x=\big(\frac{\partial}{\partial x_1},...,\frac{\partial}{\partial x_d}\big)^T$ and the Hessian operator $\nabla^2_x=\nabla_x\otimes\nabla_x=\big(\frac{\partial^2}{\partial x_j\partial x_k}\big)_{j,k=1,...,d}$, where $\otimes$ means the tensor product. The variables or subscripts will be omitted when there is no ambiguity. Then we can rewrite the operator $L$ as
\begin{equation}\label{ewo1}
  \begin{split}
     Lu =&\ Au+Ku \\
       =& \ \text{tr}[\mathbf{a}^T(\nabla^2 u)]+b^T\nabla u+cu \\
       &\ +\int_{\R^d\setminus\{0\}}\big[u(\cdot+z)-u-z^T\nabla u\cdot\mathbf{1}_{\{|z|<1\}}\big]\nu(\cdot,dz),
  \end{split}
\end{equation}
where ``tr'' denote the trace of a matrix.  Both $x^Ty$ and $x\cdot y$, for two vectors $x,y\in\R^d$,   denote the scalar product. Moreover, we denote the positive and negative part of function $u$ by $u^+:=u\lor0$ and $u^-:=-(u\land0)=(-u)\lor0$, respectively. Then $u=u^+-u^-$ and $|u|=u^++u^-$.

In this section, $L$ is the elliptic Waldenfels operator as defined in (\ref{ewo}).

\subsection{Weak maximum principle for elliptic case}
We now prove the weak maximum principle.

\begin{Theorem}[Weak maximum principle for elliptic Waldenfels operators]\label{WMPe}
 Let $D$ be an open and bounded set but not necessarily connected,  and $E$ be an open set satisfying $D\subset E$.  Assume that $u\in C^2(D)\cap C(\overline{E})$, $Lu\ge0$ in $D$, and $\text{\upshape supp}\,\nu(x,\cdot)\subset\overline E-x$ for each $x\in D$.
  \begin{enumerate}
    \item If $c\equiv0$ in $D$, then $$\sup_{\overline{E}}u=\sup_{\overline{E}\setminus D}u.$$
    \item If $c\le0$ in $D$, then
        \begin{equation*}
          \sup_{\overline{E}}u\le\sup_{\overline{E}\setminus D}u^+.
        \end{equation*}
  \end{enumerate}
Here the supremum may be infinity.
\end{Theorem}
\begin{proof}
Assertion 1.   We first consider  the case with the strict inequality
\begin{equation}\label{Lu1}
  Lu>0\quad\text{in }D.
\end{equation}
Suppose that on the contrary $\sup_{\overline{E}}u>\sup_{\overline{E}\setminus D}u$. Then   there exists a point $x^0\in D$ with $u(x^0)=\sup_{\overline{E}}u$,
and
$$u(x^0)=\max_{\overline{D}}u.$$
Thus at the maximizer point $x^0$, we have
\begin{align}
  \nabla u(x^0) &= 0, \label{du} \\
  \nabla^2u(x^0) &\le 0, \label{d2u}
\end{align}
where the last inequality means that the symmetric matrix  $\nabla^2u(x^0)$ is nonpositive definite. In particular, $\frac{\partial^2 u}{\partial x_j^2}(x^0)\le0, j=1,...,d$.
Since the matrix $\mathbf{a}=(a_{jk})$ is symmetric and positive definite at $x^0$, there exists an orthogonal matrix $P$ such that
\begin{equation*}
  P[\mathbf{a}(x^0)]P^T=\text{diag}(\lambda_1,...,\lambda_d),
\end{equation*}
where ``diag'' means the diagonal matrix with diagonal entries $\lambda_j>0,j=1,...,d$, which are eigenvalues of $\mathbf{a}(x^0)$. Then by changing variables $y-x^0=P(x-x^0)$, we have
\begin{equation*}
  \begin{split}
     \nabla_x u &= P^T(\nabla_y u), \\
     \nabla^2_x u &= P^T(\nabla^2_y u)P.
  \end{split}
\end{equation*}
In light of (\ref{d2u}), we find that at point $x^0$,
\begin{equation}\label{a}
  \begin{split}
     &\ \text{tr}[\mathbf{a}^T(\nabla^2_x u)] = \text{tr}[\mathbf{a}^TP^T(\nabla^2_y u)P]=\text{tr}[P\mathbf{a}^TP^T(\nabla^2_y u)] \\
       = &\ \text{tr}[(P\mathbf{a}P^T)^T(\nabla^2_y u)]= \sum_{j}\lambda_j\frac{\partial^2u}{\partial y_j^2}\le0.
  \end{split}
\end{equation}
Thus, combining (\ref{du}), (\ref{a}) and the assumption $c\equiv0$, together with the fact that $u$ attains a maximum at $x^0$, we obtain that at $x^0$,
\begin{equation*}
  \begin{split}
     Au &= \text{tr}[\mathbf{a}^T(\nabla^2 u)]+b^T\nabla u+cu\le0, \\
     Ku(x^0) &= \int_{\R^d\setminus\{0\}}\big[u(x^0+z)-u(x^0)-z^T\nabla u(x^0)\cdot\mathbf{1}_{\{|z|<1\}}\big]\nu(x^0,dz) \\
       &= \int_{\overline E-x}\big[u(x^0+z)-u(x^0)\big]\nu(x^0,dz) \\
       &\le 0.
  \end{split}
\end{equation*}
Hence
\begin{equation}\label{Lu2}
  Lu=Au+Ku\le0 \quad\text{at }x^0.
\end{equation}
Therefore, we get a contradiction in light of (\ref{Lu1}) and (\ref{Lu2}), which leads to $\sup_{\overline{E}}u=\sup_{\overline{E}\setminus D}u$.

For the general case that $Lu\ge0$,  we introduce a function
\begin{equation}\label{ue_e}
  u^\epsilon(x):=u(x)+\epsilon e^{-\beta x_1},\quad x\in\overline{E},
\end{equation}
where $\beta>0$ will be selected below and $\epsilon$ is a positive  parameter. Note that $a_{11}\ge\gamma>0$,    by substituting $z=e_1=(1,0,...,0)$ into condition (\ref{pd}). Then by Taylor expansion and the moment condition (\ref{LM}) of kernel $\nu$, we have
\begin{equation*}
  \begin{split}
    Lu^\epsilon =&\ Lu+\epsilon L(e^{-\beta x_1}) \\
    \ge&\ \epsilon e^{-\beta x_1}\Big[\beta^2a^{11}-\beta b^1+\int_{|z|\ge1}\big(e^{-\beta z_1}-1\big)\nu(x,dz) \\
    &\qquad\quad\ +\int_{0<|z|<1}\big(e^{-\beta z_1}-1+\beta z_1\big)\nu(x,dz)\Big] \\
    \ge&\ \epsilon e^{-\beta x_1}\Big[\beta^2a^{11}-\beta b^1-\int_{|z|\ge1}\nu(x,dz)+\frac{1}{2} \beta^2\int_{0<|z|<1}z_1^2e^{-\beta\theta z_1}\nu(x,dz)\Big] \\
    \ge&\ \epsilon e^{-\beta x_1}\Big[\beta^2a^{11}-\beta b^1-\int_{|z|\ge1}\nu(x,dz)+\frac{1}{2} \beta^2e^{-\beta\theta}\int_{0<|z|<1}z_1^2\nu(x,dz)\Big] \\
    >&\ 0,
  \end{split}
\end{equation*}
provided $\beta>0$ is large enough, where $\theta$  is a constant with $0<\theta<1$.

Then by the previous conclusion, $\sup_{\overline{E}}u^\epsilon=\sup_{\overline{E}\setminus D}u^\epsilon$. Let $\epsilon\to0$ to find $\sup_{\overline{E}}u=\sup_{\overline{E}\setminus D}u$ by the continuity. This proves Assertion 1.

Assertion 2. If $u\le0$ everywhere in $D$, the second assertion is trivially true. Hence we set $D_+:=\{x\in D\mid u(x)>0\}\ne\emptyset$. Then
\begin{equation*}
  (L-c)u\ge-cu\ge0 \quad\text{in }D_+.
\end{equation*}
The new operator $L-c$ has no zeroth-order term and consequently Assertion 1 implies that
\begin{equation*}
  \sup_{\overline{E}}u = \sup_{\overline{E}\setminus D_+}u=\big(\sup_{\overline{E}\setminus D}u\big)\lor\big(\sup_{D\setminus D_+}u\big)=\big(\sup_{\overline{E}\setminus D}u\big)\lor0=\sup_{\overline{E}\setminus D}u^+.
\end{equation*}
This completes the proof.
\end{proof}

\begin{Remark}\label{self}
  From the proof of Assertion 2 in Theorem \ref{WMPe},  we have  the following conclusions.
  \begin{enumerate}
    \item In Assertion 1, if   $Lu>0$ in $D$, then $u$ can either achieve its (finite) maximum only on $\overline{E}\setminus D$ or be unbounded on $\overline E$.
    \item In Assertion 2, essentially the following equality holds according to the proof,
        $$\sup_{\overline{E}}u^+=\sup_{\overline{E}\setminus D}u^+,$$
        even though the   Assertion 1 in Theorem \ref{WMPe} cannot be applied directly to $u^+$ as it is not in $C^2(D)$. Especially if $u$ can take positive values in $D$, or equivalently, $D_+\neq\emptyset$, then we have
        $$\sup_{\overline{E}}u=\sup_{\overline{E}\setminus D}u^+.$$
  \end{enumerate}
\end{Remark}

\begin{Remark}\label{WMPde}
  The proof of Theorem \ref{WMPe} still works if the matrix $\mathbf{a}=(a_{jk})$ is only positive semidefinite. Indeed, since the eigenvalues of $\mathbf{a}(x^0)$ are nonnegative ($\lambda_j\ge0, j=1,...,d$), the inequality (\ref{a}) still holds.
\end{Remark}

\begin{Remark}
   As in Remark \ref{special}, there are two special cases for Theorem \ref{WMPe}, that is, $E=\R^d$ or $E=D$. Using the latter as an example, namely, $u\in C^2(D)\cap C(\overline{D})$, $Lu\ge0$ in $D$, and $\text{\upshape supp}\,\nu(x,\cdot)\subset\overline D-x$ for each $x\in D$, where $D$ is open and bounded but not necessarily connected, then the following conclusions holds:
  \begin{enumerate}
    \item If $c\equiv0$ in $D$, then $$\max_{\overline{D}}u=\max_{\partial{D}}u.$$
    \item If $c\le0$ in $D$, then
        \begin{equation*}
          \max_{\overline{D}}u\le\max_{\partial{D}}u^+.
        \end{equation*}
  \end{enumerate}
\end{Remark}

\begin{Corollary}\label{L<=}
  Let $D$ be an open and bounded set but not necessarily connected,  and $E$ be an open set satisfying $D\subset E$.   Assume that $u\in C^2(D)\cap C(\overline{E})$, and $\text{\upshape supp}\,\nu(x,\cdot)\subset\overline E-x$ for each $x\in D$.
  \begin{enumerate}
    \item If $c\equiv0$ and $Lu\le0$ both hold in $D$, then
        $$\inf_{\overline{E}}u=\inf_{\overline{E}\setminus D}u.$$
    \item If $c\le0$ and $Lu\le0$ both hold in $D$, then $$\inf_{\overline{E}}u\ge-\sup_{\overline{E}\setminus D}u^-.$$
    \item If $c\le0$ and $Lu=0$ both hold in $D$, then $$\sup_{\overline{E}}|u|=\sup_{\overline{E}\setminus D}|u|.$$
  \end{enumerate}
  In all the three expressions, the supremum and infimum may be infinity.
\end{Corollary}

\begin{proof}
  1. Apply directly the first assertion of Theorem \ref{WMPe} to $-u$.

  2. Apply the second assertion of Theorem \ref{WMPe} to $-u$.

  3. Applying Statement 2 in Remark \ref{self} to $-u$, we have
  $$\sup_{\overline{E}}u^-=\sup_{\overline{E}\setminus D}u^-.$$
Then it follows that
  \begin{equation*}
    \sup_{\overline{E}}|u|=\big(\sup_{\overline{E}}u^+\big) \lor\big(\sup_{\overline{E}}u^-\big)=\big(\sup_{\overline{E}\setminus D}u^+\big)\lor\big(\sup_{\overline{E}\setminus D}u^-\big)=\sup_{\overline{E}\setminus D}|u|.
  \end{equation*}
 This completes the proof.
\end{proof}

Going one step further, we suppose $E$ is bounded and then apply Corollary \ref{L<=} to $u-v$, yielding the following corollary which is often used in applications.
\begin{Corollary}
  Let $D$ be an open and bounded set but not necessarily connected,  and $E$ be an open set satisfying $D\subset E$.   Assume that $u, v\in C^2(D)\cap C(\overline{E})$, $c\le0$ in $D$, and $\text{\upshape supp}\,\nu(x,\cdot)\subset\overline E-x$ for each $x\in D$.
  \begin{enumerate}
    \item (Comparison Principle) If $Lu\le Lv$ in $D$ and $u\ge v$ on $\overline{E}\setminus D$, then $u\ge v$ in $\overline{E}$.
    \item (Uniqueness) If $Lu=Lv$ in $D$ and $u=v$ on $\overline{E}\setminus D$, then $u=v$ in $\overline{E}$.
  \end{enumerate}
\end{Corollary}

\begin{proof}
  The two results immediately follow by using the last two assertions of Corollary \ref{L<=} for $u-v$.
\end{proof}

\subsection{Strong maximum principle for elliptic case}

This section is devoted to the strong maximum principle for the elliptic Waldenfels operator $L$.

\begin{Theorem}[Strong maximum principle for elliptic Waldenfels operator]\label{smpe}
  Let $D$ be an open and connected set   but not necessarily bounded,  and $E$ be an open set satisfying $D\subset E$.  Assume that $u\in C^2(D)\cap C(\overline{E})$, $Lu\ge0$ in $D$, and $\text{\upshape supp}\,\nu(x,\cdot)\subset\overline E-x$ for each $x\in D$. Moreover, assume that the mapping $x\to\nu(x,\cdot)$ is continuous in $D$.  If one of the following conditions holds:
  \begin{enumerate}
    \item $c\equiv0$ in $D$ and $u$ achieves a (finite) maximum over $\overline E$ at an interior point in $D$;
    \item $c\le0$ in $D$ and $u$ achieves a (finite) nonnegative maximum over $\overline E$ at an interior point in $D$;
    \item $u$ achieves a zero maximum over $\overline E$ at an interior point in $D$,
  \end{enumerate}
  then $u$ is constant on $\overline D$.
\end{Theorem}

Before proving this theorem, let us first give some comments on it.

\begin{Remark}\label{tran}
  The propagation of maximizer point by translation of measure support mentioned in \cite{Cio,Cov} is similar in our case. That is, if the assumptions in Theorem \ref{smpe} hold, then $u$ is a constant on the set $\overline{\bigcup_{n=0}^\infty\varLambda_n}$, where $\varLambda_n$'s are defined by induction,
  \begin{equation*}
    \varLambda_0 = {x^0}, \varLambda_{n+1}= \bigcup_{x\in D\cap\varLambda_n}[\text{\upshape supp}\,\nu(x,\cdot)+x].
  \end{equation*}
  This result depends on the support of every measure $\nu(x,\cdot)$, it can be easily proved by induction and continuity. It is noteworthy that in this scheme, the set $D$ may not be connected, since jumps from one connected component to another might occur when measure supports overlap two or more connected components.

  In conclusion, it is the integro-differential term, or jump diffusion term that leads to the propagation of maximizer point between those connected components. Therefore, we need to restrict that the Markovian point can move by jumps only inside the set $E$, i.e., the support condition (\ref{supportNu2}), to obtain the propagation of maximizer (over $E$) point.
\end{Remark}

\begin{Remark}
  As shown in Remark \ref{special}, our results on the weak and strong maximum principles formulated in Theorem \ref{WMPe} and \ref{smpe}, respectively,    cover the situations when the support of jump measure is either bounded or unbounded, especially for $E=D$ or $E=\R^d$ in the setting. While   Taira \cite{KTaira} only considered the situation for $E=D$. Furthermore, our assumptions are less restrictive than Taira's: In our work,  the connectedness is not needed for the weak maximum principle while the boundedness is not necessary for the strong maximum principle. Moreover,  the continuity of mapping $x\to\nu(x,\cdot)$ is necessary only in the strong case but not for the weak maximum principle.
\end{Remark}

Like the weak case, by applying directly Theorem \ref{smpe} to $-u$, one can conclude the strong maximum principle for the converse case $Lu\le0$.

\begin{Corollary}\label{s_Lu<}
Let $D$ be an open and connected set   but not necessarily bounded,  and $E$ be an open set satisfying $D\subset E$.
  Assume that $u\in C^2(D)\cap C(\overline{E})$, $Lu\le0$ in $D$, and $\text{\upshape supp}\,\nu(x,\cdot)\subset\overline E-x$ for each $x\in D$. Moreover, assume that the mapping $x\to\nu(x,\cdot)$ is continuous in $D$. If one of the following conditions holds:
  \begin{enumerate}
    \item $c\equiv0$ in $D$ and $u$ achieves a (finite) minimum over $\overline E$ at an interior point in $D$;
    \item $c\le0$ in $D$ and $u$ achieves a (finite) nonnegative minimum over $\overline E$ at an interior point in $D$;
    \item $u$ achieves a zero minimum over $\overline E$ at an interior point in $D$,
  \end{enumerate}
  then $u$ is constant on $\overline D$.
\end{Corollary}

\vspace{3mm}
Now we start to prove Theorem \ref{smpe}.

\begin{proof}[Proof of Theorem \ref{smpe}]
  Suppose that $u\not\equiv\max_{\overline E}u$ in $D$. Set $D_<:=\{x\in D\mid u(x)<\max_{\overline E}u\}\ne\emptyset$. Since $D$ is connected which implies $\partial D_<\cap D\ne\emptyset$, we can always choose a point $x^1\in D_<$ such that $\text{dist}(x^1,\partial D_<\cap D)<\text{dist}(x^1,\partial D)$. Denote by $B$ the largest ball having $x^1$ as center with $B\subset D_<$. Then $\overline{B}\subset D$ and there exists some point $x^0\in\partial B$ with
\begin{equation*}
  u(x^0)=\max_{\overline E}u>u(x),\quad\forall x\in B.
\end{equation*}
Since $u$ achieves its maximum at $x^0\in D$, we have $\nabla u(x^0)=0$. We will create a contradiction by proving that
\begin{equation}\label{contradiction}
  \frac{\partial u}{\partial {\bf n}}(x^0)>0,
\end{equation}
where ${\bf n}$ is the unit outer normal vector of $B$ at $x^0$. Then by this contradiction, $u$ must be constant within $D$, and the result follows by continuity. Now the rest of the proof is devoted to \eqref{contradiction}. We divide it into three steps.

\vspace{1.5mm}
\textit{Step 1.} The closed set $\overline{B}$ is a $d$-dimension $C^2$-differential manifold with boundary. Let $(U,\mathbf{\Phi})$ be a coordinate chart near $x^0$, where $U$ is a relatively open neighborhood of $x^0$ in $\overline{B}$, $\mathbf{\Phi}$ is a $C^2$-diffeomorphism to its image from $U$ into the closed upper half plane $\mathbb{H}^d_+:=\{y\in\mathbb{R}^d\mid y_d\ge0\}$, with inverse $\mathbf{\Phi}^{-1}$. Then $\mathbf{\Phi}$ is an embedding whose rank at $x^0$ equals to $d$, equivalently, if we denote by $J\mathbf{\Phi}$ the Jacobian matrix of $\mathbf{\Phi}$, i.e., $J\mathbf{\Phi}:=\nabla_x\mathbf{\Phi}$, then $J\mathbf{\Phi}$ is non-degenerate. As a result, the tangent mapping $\mathbf{\Phi}_\ast$ induced by $\mathbf{\Phi}$ at point $x^0$ is an isomorphism.

Now we consider the function $u$ restricted in $U$. We define    $\hat{u}(y):=u(\mathbf{\Phi^{-1}}(y)), y\in\mathbf{\Phi}(U)$. Then $\hat{u}$ attains its maximum at $y^0=\mathbf{\Phi}(x^0)$ over $\mathbf{\Phi}(U)\subset\mathbb{H}^d_+$. Hence at the maximizer point $y^0$,
\begin{equation}\label{d-1}
  \frac{\partial\hat{u}}{\partial y_j}=0,\quad j=1,...,d-1.
\end{equation}
We also denote the image tangent vector of $\frac{\partial}{\partial\mathbf{n}}$ under tangent mapping $\mathbf{\Phi}_\ast$ by
\begin{equation*}
  \frac{\partial}{\partial\hat{\mathbf{n}}}:= \mathbf{\Phi}_\ast\Big(\frac{\partial}{\partial\mathbf{n}}\Big).
\end{equation*}
We compute at $y^0$ (or    $x^0$)
\begin{equation}\label{u_n}
  \begin{split}
     &\ \frac{\partial\hat{u}}{\partial\hat{\mathbf{n}}} = \bigg\langle\mathbf{\Phi}_\ast\Big(\frac{\partial} {\partial\mathbf{n}}\Big),d\hat{u}\bigg\rangle= \bigg\langle\frac{\partial} {\partial\mathbf{n}},\mathbf{\Phi}^\ast(d\hat{u})\bigg\rangle \\
       =&\ \bigg\langle\frac{\partial}{\partial\mathbf{n}}, d(\hat{u}\circ\mathbf{\Phi})\bigg\rangle= \bigg\langle\frac{\partial}{\partial\mathbf{n}},du\bigg\rangle = \frac{\partial u}{\partial\mathbf{n}}=0,
  \end{split}
\end{equation}
where $\mathbf{\Phi}^\ast$ is denoted as the cotangent mapping induced by $\mathbf{\Phi}$ at point $x^0$, $\langle\cdot,\cdot\rangle$ is the dual product between the tangent space and cotangent space at $y^0$ (or $x^0$). Now recall  that $\mathbf{\Phi}_\ast$ is an isomorphism. The tangent vector $\frac{\partial}{\partial\hat{\mathbf{n}}}$ is independent of $\{\frac{\partial}{\partial y_j}\mid j=1,...,d-1\}$ and consequently by (\ref{u_n}),
\begin{equation}\label{d}
  \frac{\partial\hat{u}}{\partial y_d}(y^0)=0.
\end{equation}
Combining (\ref{d-1}) and (\ref{d}) together with the fact that $\hat{u}$ attains its maximum at $y^0$, we have
\begin{equation}\label{d2}
  \nabla_y^2\hat{u}(y^0)\le0.
\end{equation}

Combining (\ref{d-1}), (\ref{d}) and (\ref{d2}), we have at $x^0$,
\begin{align*}
  \nabla_x u &= (J\mathbf{\Phi})^T(\nabla_y\hat{u})=0, \\
  \nabla_x^2 u &= (J\mathbf{\Phi})^T(\nabla_y^2\hat{u})(J\mathbf{\Phi}) +(\nabla^2_x\mathbf{\Phi})(\nabla_y\hat{u})= (J\mathbf{\Phi})^T(\nabla_y^2\hat{u})(J\mathbf{\Phi}),
\end{align*}
where we treat $\nabla^2_x\mathbf{\Phi}$ as a third-order covariant tensor. Hence at $x^0$,
\begin{equation}\label{Au}
  \begin{split}
     Au &= \text{tr}[\mathbf{a}^T(\nabla^2_x u)]+b^T\nabla_x u+cu \\
       &= \text{tr}[\mathbf{a}^T(J\mathbf{\Phi})^T(\nabla_y^2 \hat{u})(J\mathbf{\Phi})]+cu \\
       &= \text{tr}[(J\mathbf{\Phi})\mathbf{a}^T (J\mathbf{\Phi})^T(\nabla_y^2 \hat{u})]+cu \\
       &= \text{tr}\big[\big((J\mathbf{\Phi})\mathbf{a}(J\mathbf{\Phi})^T\big)^T (\nabla_y^2\hat{u})\big]+cu \\
       &=: \text{tr}[\hat{\mathbf{a}}^T(\nabla^2_y\hat{u})]+cu,
  \end{split}
\end{equation}
where $\hat{\mathbf{a}}:=(J\mathbf{\Phi})\mathbf{a}(J\mathbf{\Phi})^T$. Since $\mathbf{a}(x)$ is symmetric and positive definite and the matrix $J\mathbf{\Phi}$ is non-degenerate, we see the matrix $\hat{\mathbf{a}}(x^0)$ is also symmetric and positive definite.
Hence, as explained in the proof of Theorem \ref{WMPe} and by   (\ref{d2}),  we have
\begin{equation}\label{ahat}
  \text{tr}\big[\hat{\mathbf{a}}(x^0)^T\big(\nabla^2_y\hat{u}(y^0)\big)\big] \le0.
\end{equation}

Define
\begin{equation}\label{E0}
  E_0:=\big\{x\in \overline E\mid u(x)=\max_{\overline E}u\big\}=\big\{x\in\overline E\setminus B\mid u(x)=\max_{\overline E}u\big\}.
\end{equation}
Recall that $u$ attains its maximum over $\overline{E}$ at $x^0$. Now we have
\begin{equation}\label{Ku}
  \begin{split}
     Ku(x^0) &= \int_{\mathbb{R}^d\setminus\{0\}}\big[u(x^0+z)-u(x^0)- z^T\nabla u(x^0)\cdot\mathbf{1}_{\{|z|<1\}}\big]\nu (x^0,dz) \\
       &= \int_{\overline E}[u(x^0+z)-u(x^0)]\nu (x^0,dz) \\
       &= \int_{x^0+z\in\overline E\setminus E_0}[u(x^0+z)-u(x^0)]\nu (x^0,dz) \\
       &\le 0.
  \end{split}
\end{equation}

From (\ref{Au}), (\ref{ahat}) and (\ref{Ku}), we obtain
$$Lu(x^0)=Ku(x^0)+Au(x^0)\le c(x^0)u(x^0)\le 0.$$
By recalling the assumption on $u$, we have $Lu(x)\ge 0$ for each $x\in D$, and thus
\begin{equation*}
  Lu(x^0)=Au(x^0)=Ku(x^0)=0,
\end{equation*}
especially,
\begin{equation*}
  Ku(x^0) = \int_{x^0+z\in\overline E\setminus E_0}[u(x^0+z)-u(x^0)]\nu (x^0,dz) = 0.
\end{equation*}
Hence, we conclude
\begin{equation}\label{nu}
  \nu(x^0,(\overline E\setminus E_0)-x^0)=0.
\end{equation}

\vspace{1.5mm}
\textit{Step 2.} We set $B=B(x^1,R)$ with $R=|x^0-x^1|$. See Figure \ref{fig:Hopf}. Define
$$v(x):=e^{-\beta|x-x^1|^2}-e^{-\beta R^2},\quad x\in\overline{E},$$
for $\beta>0$ as selected below. Then

\begin{figure}[!bth]
  \centering
  \includegraphics[width=\textwidth]{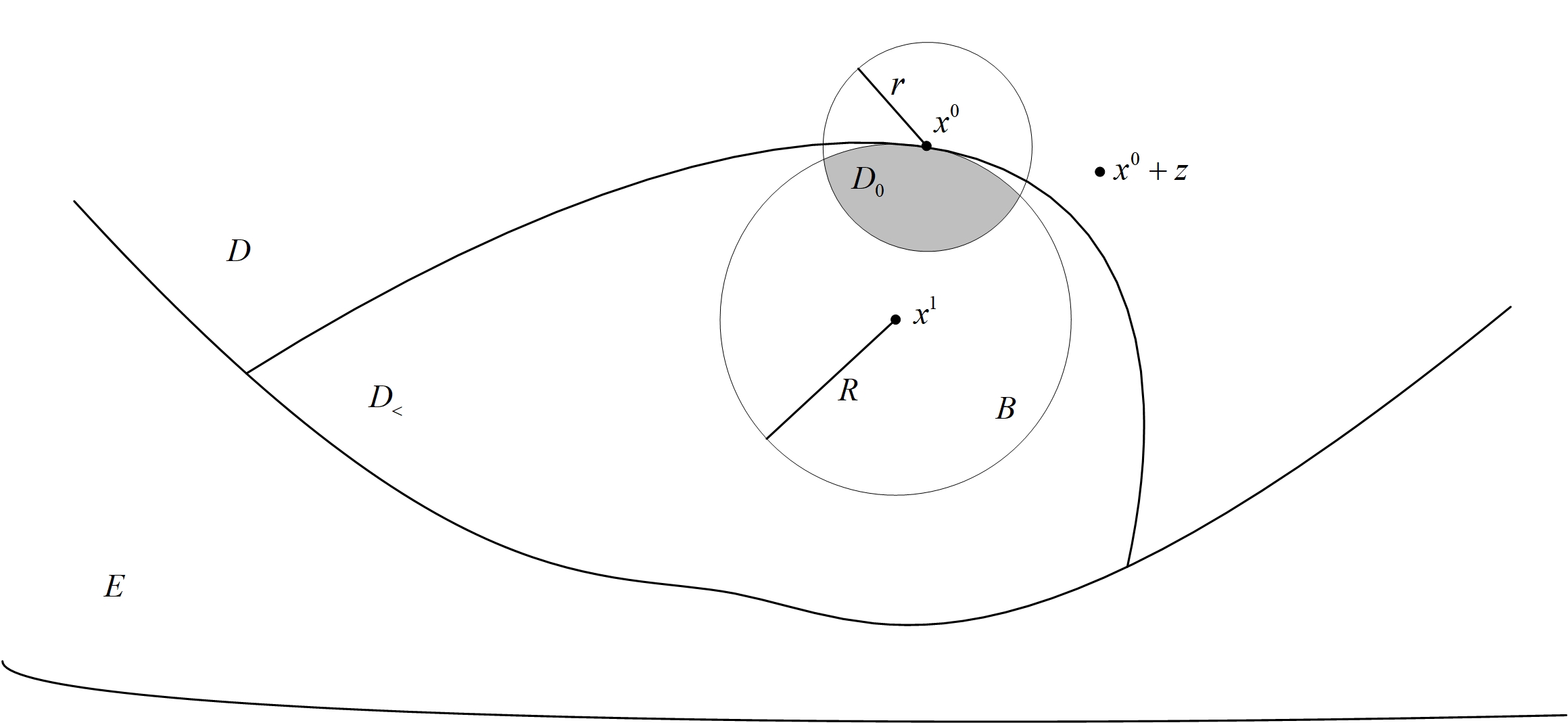}
  \caption{Sketch for Theorem \ref{smpe}.} \label{fig:Hopf}
\end{figure}

\begin{equation}\label{Av1}
  \begin{split}
     Av =&\ e^{-\beta|x-x^1|^2}\Big\{\text{tr}\big[\mathbf{a}^T \big(4\beta^2(x-x^1)\otimes(x-x^1)-2\beta I\big)\big] \\
       &\qquad\qquad\ \ -2\beta b^T(x-x^1)+c\big(1-e^{-\beta(R^2-|x-x^1|^2)}\big)\Big\} \\
       =&\ e^{-\beta|x-x^1|^2}\big[4\beta^2(x-x^1)^T\mathbf{a}^T (x-x^1)-2\beta\text{tr}(\mathbf{a}) \\
       &\qquad\qquad\ \ -2\beta b^T(x-x^1)+c\big(1-e^{-\beta(R^2-|x-x^1|^2)}\big)\big] \\
       \ge&\ e^{-\beta|x-x^1|^2}\big[4\gamma\beta^2|x-x^1|^2-2\beta\textrm{tr} (\mathbf{a})-2\beta|b||x-x^1| \\
       &\qquad\qquad\ \ +c\big(1-e^{-\beta(R^2-|x-x^1|^2)}\big)\big].
  \end{split}
\end{equation}
Consider next the open set $D_0:=B(x^1,R)\cap B(x^0,r)$ (see Figure \ref{fig:Hopf}) with some $r\in(0,R)$ which will be chosen later. When $\beta$ is large enough, we have
\begin{equation}\label{Av2}
  Av \ge e^{-\beta R^2}\big[4\gamma\beta^2(R-r)^2-2\beta\text{tr} (\mathbf{a})-2\beta|b|R\big] > C_1\beta^2-C_2\beta,
\end{equation}
for $x\in D_0$, where $C_1,C_2$ are two positive constants.

Moreover, by recalling (\ref{nu}), we have
\begin{equation}\label{Kv}
  \begin{split}
     Kv(x^0) =& \int_{\mathbb{R}^d\setminus{\{0\}}}\big[e^{-\beta|x^0-x^1+z|^2} -e^{-\beta R^2} \\
       &\qquad\quad\,\ +2\beta z^T(x^0-x^1)e^{-\beta R^2}\mathbf{1}_{\{|z|<1\}}\big]\nu(x^0,dz) \\
       =& \int_{\substack{x^0+z\in E_0 \\ |z|\ge1}}\big[e^{-\beta|x^0-x^1+z|^2} -e^{-\beta R^2}\big]\nu(x^0,dz) \\
       &+ \int_{\substack{x^0+z\in E_0 \\ 0<|z|<1}}\big[e^{-\beta|x^0-x^1+z|^2} -e^{-\beta R^2} \\
       &\qquad\qquad\quad\ +2\beta z^T(x^0-x^1)e^{-\beta R^2}\big]\nu(x^0,dz) \\
       =:&\ \uppercase\expandafter{\romannumeral1} +\uppercase\expandafter{\romannumeral2}.
  \end{split}
\end{equation}

For the term $\uppercase\expandafter{\romannumeral1}$, it is clear that $E_0\cap\overline B=\{x^0\}$ and consequently
\begin{equation*}
  |x^0+z-x^1|>|x^0-x^1|=R
\end{equation*}
for point $z$ satisfying $x^0+z\in\overline E\setminus E_0$. Thus for sufficiently large $\beta$, we have
\begin{equation*}
  -C_3<e^{-\beta|x^0-x^1+z|^2}-e^{-\beta R^2}<0
\end{equation*}
with a constant $C_3>0$. Hence,
\begin{equation}\label{I1}
  \uppercase\expandafter{\romannumeral1} > -C_3\int_{\substack{x^0+z\in E_0 \\ |z|\ge1}}\nu(x^0,dz)\gtrsim-C_3.
\end{equation}

For the term $\uppercase\expandafter{\romannumeral2}$, using the Taylor expansion,  and for $x^0+z\in E_0$ and $\beta$   large enough,
\begin{equation*}
  \begin{split}
     &\ e^{-\beta|x^0-x^1+z|^2} -e^{-\beta R^2}+2\beta z^T(x^0-x^1)e^{-\beta R^2} \\
     =&\ \frac{1}{2} \big[4\beta^2e^{-\beta|x^0-x^1+\theta z|^2}| z^T(x^0-x^1+\theta z)|^2-2\beta e^{-\beta|x^0-x^1+\theta z|^2}|z|^2\big] \\
     \ge& -\beta e^{-\beta|x^0-x^1+\theta z|^2}|z|^2 \\
     \ge& -C_4\beta|z|^2,
  \end{split}
\end{equation*}
with some $\theta\in(0,1)$ and a constant $C_4>0$. Hence,
\begin{equation}\label{I2}
  \uppercase\expandafter{\romannumeral2} > -C_4\beta\int_{\substack{x^0+z\in E_0 \\ 0<|z|<1}}|z|^2\nu(x,dz)\gtrsim-C_4\beta.
\end{equation}

Thus, combining the results of (\ref{Av2}), (\ref{Kv}), (\ref{I1}) and (\ref{I2}), we find that
\begin{align*}
  Lv(x^0) &= Av(x^0)+Kv(x^0)=Av(x^0)+\uppercase\expandafter{\romannumeral1} +\uppercase\expandafter{\romannumeral2} \\
  &\gtrsim C_1\beta^2-(C_2+C_4)\beta-C_3 \\
  &> 0,
\end{align*}
provided $\beta>0$ is fixed large enough. Since $Lv(x)$ is continuous in $x\in D$ in light of the continuity of $\nu(x,\cdot)$, we have
\begin{equation}\label{Lv}
  Lv(x)\ge0,
\end{equation}
for $x\in D_0$, provided $r$ is small enough.

\vspace{1.5mm}
\textit{Step 3.} Define
\begin{equation*}
  u^\epsilon(x)=u(x)+\epsilon v(x)-u(x^0),\quad x\in\overline{E},
\end{equation*}
for a constant $\epsilon>0$.
We can choose $\epsilon$ so small that
\begin{equation*}
  u^\epsilon(x)\le0,\quad x\in\overline{E}\setminus D_0,
\end{equation*}
since $v(x)\le0$ for $x\in\overline{E}\setminus B$, and $u(x)<u(x^0)$ for $x\in B\setminus D_0$ by recalling $u(x^0)>u(x)$ for all $x\in D_<$.

For the first two cases, $c\equiv0$ in $D$, or $c\le0$ in $D$ also $u(x^0)\ge0$, from (\ref{Lv}) and the fact that $Lu\ge0$ in $D$, we see that
\begin{equation*}
  Lu^\epsilon\ge-cu(x^0)\ge0 \quad\text{in }D_0.
\end{equation*}
In view of the weak maximum principle of elliptic Waldenfels operator, Theorem \ref{WMPe},   we know that $u^\epsilon\le 0$ in $\overline{E}$. Note that $u^\epsilon(x^0)=0$. Thus we have,
\begin{equation*}
  0=\frac{\partial u^\epsilon}{\partial\mathbf{n}}(x^0)=\frac{\partial u}{\partial\mathbf{n}}(x^0)+\epsilon\frac{\partial v}{\partial\mathbf{n}}(x^0).
\end{equation*}
Consequently,
\begin{equation*}
  \frac{\partial u}{\partial\mathbf{n}}(x^0)=-\epsilon\frac{\partial v}{\partial\mathbf{n}}(x^0)=-\epsilon\nabla v(x^0)\cdot\frac{(x^0-x^1)}{R}=2\epsilon\beta Re^{-\beta R^2}>0,
\end{equation*}
as required.

For the third case that $u(x^0)=0$, obviously $u\le0$ in $D$. We find
\begin{equation*}
  (L-c^+)u=Lu-c^+u\ge Lu\ge0 \quad\text{in }D.
\end{equation*}
Notice that the zeroth-order coefficient of operator $L-c^+$ is $c-c^+$, which is nonpositive in $D$. Hence we apply the result of the second case by replacing $L$ and $c$ respectively with $L-c^+$ and $c-c^+$ to get the same result for this case.

We have thus completed the proof.
\end{proof}

Some comments will be helpful for understanding the long proof of Theorem \ref{smpe}.
\begin{Remark}\label{component}
  In Theorem \ref{smpe}, we restrict the set $D$ to be connected to ensure $\partial D_<\cap D\ne\emptyset$. More generally, if $D$ is not connected, one may merely replace $D$ with the connected component of $D$ which contains the maximizer point, and we thus conclude that $u$ is constant in this connected component.

Recalling Remark \ref{lead}, we could see that the diffusion term gives rise to the propagation of maximizer point in the  corresponding connected component. This is why we need the set $D$ to be connected.
\end{Remark}

\begin{Remark}\label{lead}
  We can see from (\ref{Av2}), (\ref{I1}) and (\ref{I2}) that, it is the second-order differential term $\textrm{\upshape tr}[\mathbf{a}^T(\nabla^2)]$, namely, the diffusion term that plays a leading role in Step 2 in the proof of Theorem \ref{smpe}.
\end{Remark}

\begin{Remark}\label{psd}
  Theorem \ref{smpe} still holds if the matrix $\mathbf{a}(x)=(a_{jk}(x))_{j,k=1,...,d}$ is only positive semidefinite and the unit outer normal vector $\mathbf{n}$ is not in the nullspace of $\mathbf{a}(x^0)$.

  In fact, recall that $\hat{\mathbf{a}}=(J\mathbf{\Phi})\mathbf{a}(J\mathbf{\Phi})^T$ is also semidefinite as the Jacobian matrix $J\mathbf{\Phi}$ is invertible. Due to the reason mentioned in Remark \ref{WMPde}, we confirm that (\ref{ahat}) still holds. Moreover, noting that there exists a positive constant $\gamma$ such that $\mathbf{n}^Ta(x^0)\mathbf{n}\ge\gamma>0$ with $\mathbf{n}$ not in the nullspace of $\mathbf{a}(x^0)$, and consequently $(x^0-x^1)^Ta(x^0)(x^0-x^1)\ge\gamma|x^0-x^1|^2$. By   continuity we can choose $r$  so small that for all $x\in D_0=B(x^1,R)\cap B(x^0,r)$,
  \begin{equation*}
    (x-x^1)^Ta(x)(x-x^1)\ge\gamma_1|x-x^1|^2,
  \end{equation*}
  with a positive constant $\gamma_1$. Hence (\ref{Av1}) holds with $\gamma_1$ in placing of $\gamma$ and (\ref{Av2}) also holds for some other   constants $C_1, C_2$.
\end{Remark}

By a similar way to prove \eqref{contradiction}, we can easily obtain the following version of Hopf's boundary point lemma, which is a generalization of \cite[Lemma C.3]{KTaira}.

\begin{Proposition}[Hopf's boundary point lemma for elliptic Waldenfels operators]\label{Hopf}
   Let $D$ be an open set (not necessarily bounded or connected) with boundary $\partial D$ being $C^2$. Assume that $u\in C^2(\overline{D})$, $Lu\ge0$ in $D$, and $\text{\upshape supp}\,\nu(x,\cdot)\subset\overline D-x$ for each $x\in D$, and furthermore  the mapping $x\to\nu(x,\cdot)$ is continuous in $D$. Suppose that $u$ achieves its (finite)  maximum over $\overline{D}$ at point $x^0\in\partial D$ such that $u(x^0)>u(x)$ for all $x\in D$, and  that    one of the following conditions holds:
\begin{enumerate}
  \item $c\equiv0$ in $D$;
  \item $c\le0$ in $D$ and $u(x^0)\ge0$;
  \item $u(x^0)=0$.
\end{enumerate}
Then the outer normal derivative is positive: $\frac{\partial u}{\partial {\bf n}}(x^0)>0$.
\end{Proposition}

In fact, if we let $E=D$ and replace $B$ by $D$ in Step 1 in the proof of Theorem \ref{smpe}, also replace $D_<$ by $D$ in Step 3, then the three-step argument also works in the context of Proposition \ref{Hopf} and the result follows.

\section{Maximum principles for parabolic Waldenfels operators}\label{mpp}

We assume   that $D$, $E$ are two open sets in $\R^d$  and $D\subset E$, where $E$ is not necessarily bounded. Set $D_T:=D\times(0,T]$ and $E_T:=E\times(0,T]$ for arbitrarily fixed $T>0$.

As in
\cite{Komatsu73,Komatsu00,Str,Stroock,Kolokoltsov}, we define a time dependent elliptic Waldenfels operator
\begin{equation}\label{WoperatorTime}
  L := A+K,
\end{equation}
where $A$ and $K$ are defined as, respectively
\begin{equation*}
  \begin{split}
     Au(x,t) &:= \sum^d_{j,k=1}a_{jk}(x,t)\frac{\partial^2 u}{\partial x_j\partial x_k}(x,t)+\sum^d_{j=1}b_{j}(x,t)\frac{\partial u}{\partial x_j}(x,t)+c(x,t)u(x,t), \\
     Ku(x,t) &:= \int_{\R^d\setminus\{0\}}\Big[u(x+z,t)-u(x,t)-\sum^d_{j=1}z_j \frac{\partial u}{\partial x_j}(x,t)\mathbf{1}_{\{|z|<1\}}\Big]\nu(t,x,dz).
  \end{split}
\end{equation*}
  We make the following assumptions:
\begin{enumerate}
  \item Continuity condition: $a_{jk},b_j,c\in C(\overline{E_T})$ $(j,k=1,...,d).$
  \item Symmetry condition: $a_{jk}=a_{kj}$ $(j,k=1,...,d)$.  \\
  Uniform ellipticity condition:
      There exists a constant $\gamma>0$ such that
      \begin{equation*}
        \sum^d_{j,k=1}a_{jk}(x,t)\xi_j\xi_k\ge\gamma|\xi|^2,
      \end{equation*}
      for all $(x,t)\in D_T$, $\xi\in\R^d$.
  \item L\'evy measures: The kernel $\{\nu(t,x,\cdot)\mid (x,t)\in\R^d\times[0,T]\}$ is a family of L\'evy measures, namely, each $\nu(t,x,\cdot)$ is a Borel measure on $\mathbb{R}^d\setminus\{0\}$ such that for all $(x,t)\in\R^d\times[0,T]$,
      \begin{equation}\label{LMt}
        \int_{\mathbb{R}^d\setminus\{0\}}(1\land|z|^2)\nu(t,x,dz)<\infty,
      \end{equation}
      and moreover, for fixed $U\in\mathcal{B}(\mathbb{R}^d\setminus\{0\})$, the mapping $\R^d\times[0,T]\ni(x,t)\to\nu(t,x,U)\in[0,\infty)$ is Borel measurable. Here we further assume that for each $(x,t)\in D_T$,  the measure $\nu(t,x,\cdot)$ is supported in $\overline E-x:=\{y-x\mid y\in\overline E\}=\{z\mid x+z\in\overline E\}$. That is,
      \begin{equation}\label{supportNu}
        \text{\upshape supp}\,\nu(t,x,\cdot)\subset\overline E-x,\quad \forall (x,t)\in D_T.
      \end{equation}
\end{enumerate}

The Markov process associated with such a generator $L$ can be determined as a solution to the martingale problem
induced by $L$ (see, e.g., \cite{Stroock}). However, it is not clear if the Markov process determined by the martingale
problem is linked to a stochastic differential equation   with certain boundary conditions.

Now we consider the  parabolic Waldenfels operator
  $$-\frac{\partial}{\partial t}+L,$$
 with $L$ being defined in (\ref{WoperatorTime}), and we are concerned with the maximum principles for such a parabolic operator.

\subsection{Weak maximum principle for parabolic case}

We are in the position to present both weak and strong maximum principles for parabolic Waldenfels operator $-\frac{\partial}{\partial t}+L$. First we prove the weak one.

\begin{Theorem}[Weak maximum principle for parabolic Waldenfels operators]\label{WMPp}
Let $D$ be an open and bounded set but not necessarily connected, and $E$ be an open set satisfying $D\subset E$.
  Assume that $u\in C^{2,1}(D_T)\cap C(\overline{E_T})$, $-\frac{\partial u}{\partial t}+Lu\ge0$ in $D_T$, and $\text{\upshape supp}\,\nu(t,x,\cdot)\subset\overline E-x$ for each $(x,t)\in D_T$.
  \begin{enumerate}
    \item If $c\equiv0$ in $D_T$, then $$\sup_{\overline{E_T}}u=\sup_{\overline{E_T}\setminus D_T}u.$$
    \item If $c\le0$ in $D_T$, then $$\sup_{\overline{E_T}}u\le\sup_{\overline{E_T}\setminus D_T}u^+.$$
  \end{enumerate}
  Here the supremum may  be infinity.
\end{Theorem}

\begin{proof}
Assertion 1. We prove this by contradiction. Suppose that the strict inequality holds, i.e.,
\begin{equation}\label{strict}
  -\frac{\partial u}{\partial t}+Lu>0 \quad\text{in }D_T,
\end{equation}
but there exists a point $(x^0,t^0)\in D_T$ such that
\begin{equation*}
  u(x^0,t^0)=\max_{\overline{E_T}}u.
\end{equation*}

On one hand, as explained in the proof of Theorem \ref{WMPe}, we note that $Lu\le0$ at point $(x^0,t^0)$. On the other hand, if $0<t^0<T$, then $(x^0,t^0)\in(D_T)^\circ$ and consequently
\begin{equation*}
  \frac{\partial u}{\partial t}=0 \quad\text{at }(x^0,t^0);
\end{equation*}
if $t^0=T$, then $(x^0,t^0)\in\partial(D_T)$ and consequently
\begin{equation*}
  \frac{\partial u}{\partial t}\ge0 \quad\text{at }(x^0,t^0).
\end{equation*}
Thus we always have $-\frac{\partial u}{\partial t}+Lu\le0$ at point $(x^0,t^0)$, a contradiction to (\ref{strict}).

In the general case that $-\frac{\partial u}{\partial t}+Lu\ge0$ holds in $D_T$, define
\begin{equation}\label{ue_p}
  u^\epsilon(x,t):=u(x,t)-\epsilon t\quad\text{in }\overline{E_T},
\end{equation}
with a positive parameter $\epsilon$. Then
\begin{equation*}
  -\frac{\partial u^\epsilon}{\partial t}+Lu^\epsilon=-\frac{\partial u}{\partial t}+Lu+\epsilon>0,
\end{equation*}
and hence $\sup_{\overline{E_T}}u^\epsilon=\sup_{\overline{E_T}\setminus D_T}u^\epsilon$. Now Assertion 1 follows by setting $\epsilon\to0$.

Assertion 2. If $u$ is nonpositive throughout $D$, Assertion 2 is trivially true. Hence we may assume on the contrary that $u$ achieves a positive maximum at a point $(x^0,t^0)\in D_T$ over $\overline{E_T}$.

We   first consider the   case with  strict inequality $-\frac{\partial u}{\partial t}+Lu>0$ in $D_T$. Since $u(x^0,t^0)>0$ and $c\le0$, we derive the contradiction to Assertion 1,
\begin{equation*}
  -\frac{\partial u}{\partial t}+Lu\le0 \quad\text{at } (x^0,t^0).
\end{equation*}

More generally,  if $-\frac{\partial u}{\partial t}+Lu\ge0$ in $D_T$, then set as before $u^\epsilon(x,t):=u(x,t)-\epsilon t$ with $\epsilon>0$, which leads to
\begin{equation*}
  -\frac{\partial u^\epsilon}{\partial t}+Lu^\epsilon=-\frac{\partial u}{\partial t}+Lu+\epsilon-c\cdot\epsilon t\ge\epsilon-c\cdot\epsilon t>0.
\end{equation*}
Moreover, if $u$ achieves a positive maximum at a point $(x^0,t^0)\in D_T$ over $\overline{E_T}$, then by the continuity, $u^\epsilon$ also achieves a positive maximum at a point $(x^0,t^0)\in D_T$ over $\overline{E_T}$,  provided that $\epsilon$ is small enough. However, as in the previous proof, we obtain a contradiction.

This completes the proof.
\end{proof}

\begin{Remark}
   As in the first statement of Remark \ref{self}, we also conclude that in Assertion 1 of Theorem \ref{WMPp} if strictly $-\frac{\partial u}{\partial t}+Lu>0$ in $D$, then $u$ can either achieve its (finite) maximum only on $\overline{E_T}\setminus D_T$ or be unbounded on $\overline{E_T}$.
\end{Remark}

\begin{Remark}
  We cannot prove Assertion 2 of Theorem \ref{WMPp} in the same way as the corresponding assertion in Theorem \ref{WMPe}. In fact, if we introduce  similarly the set $D_T^+:=\{(x,t)\in D_T\mid u(x,t)>0\}$, it will never be the form of $U\times(0,T]$ for some $U\subset D$. Hence we may not take advantage of the first assertion of Theorem \ref{WMPp}. Consequently, the similar judgment with Assertion 2 of Remark \ref{self}, which lies on the proof of Assertion 2 in Theorem \ref{WMPe}, cannot be established here.
\end{Remark}

\begin{Remark}\label{natural}
  From Theorem \ref{WMPe} and Remark \ref{WMPde}, we have already known that, for $u\in C^{2,2}(D_T)\cap C(\overline{E_T})$, the supremum (or respectively, positive supremum) is achieved on $\overline{E_T}\setminus(D_T)^\circ$. The alert reader   could  notice that we may have appeared to be cheating here, as we should also verify that the kernel $\nu$ still satisfy the third assumption in the definition of elliptic Waldenfels operator (\ref{ewo}) when regarding it as a kernel in $\R^{d+1}$. In fact,     the  modified kernel   $\hat{\nu}((x,t),dzds):=\nu(t,x,dz)\delta_0(ds)$  does satisfy the moment condition (\ref{LM}), which is enough for us even though $\hat{\nu}$ is not supported inside $\R^{d+1}\setminus\{0\}$. See the proof of Lemma \ref{step1} for details.

  Moreover, if $-\frac{\partial u}{\partial t}+Lu>0$ in $D_T$, we see that the maximum (or respectively, positive maximum) cannot be achieved on the upper boundary $D\times\{T\}$ by the same argument as in the proof of Theorem \ref{WMPp}. Hence, it is     clear that Theorem \ref{WMPp} holds for $u\in C^{2,2}(D_T)\cap C(\overline{E_T})$,  which is a natural consequence of Theorem \ref{WMPe}.

  However, the result for $u\in C^{2,1}(D_T)\cap C(\overline{E_T})$ cannot be obtained in this way. We only need the first-order differentiability in $t$, benefiting from the form of the operator $-\frac{\partial}{\partial t}+L$. This is evident    in the different forms of $u^\epsilon$ in (\ref{ue_e}) and (\ref{ue_p}).
\end{Remark}

\begin{Remark}
There are two special cases for the weak maximum principle Theorem \ref{WMPp} for the parabolic operator $-\frac{\partial}{\partial t}+L$. That is, $E=\R^d$ or $E=D$. Take the latter as an example. Let $u\in C^{2,1}(D_T)\cap C(\overline{D_T})$, $-\frac{\partial u}{\partial t}+Lu\ge0$ in $D_T$, and $\text{\upshape supp}\,\nu(t,x,\cdot)\subset\overline D-x$ for each $(x,t)\in D_T$, where $D$ is open and bounded but not necessarily connected.
  \begin{enumerate}
    \item If $c\equiv0$ in $D_T$, then $$\max_{\overline{D_T}}u=\max_{\Gamma_T}u.$$
    \item If $c\le0$ in $D_T$, then $$\max_{\overline{D_T}}u\le\max_{\Gamma_T}u^+.$$
  \end{enumerate}
Here $\Gamma_T$ is     the parabolic boundary of $D_T$, i.e., $\Gamma_T:=\overline{D_T}\setminus D_T=(\partial D\times[0,T])\cup(D\times\{0\})$.
\end{Remark}

There are   some consequences of the weak maximum principle for a parabolic Waldenfels operator. We only highlight  the following results.

\begin{Corollary}\label{p<=}
  Let $D$ be an open and bounded set but not necessarily connected, and $E$ be an open set satisfying $D\subset E$.  Assume that $u\in C^{2,1}(D_T)\cap C(\overline{E_T})$, and $\text{\upshape supp}\,\nu(t,x,\cdot)\subset\overline E-x$ for each $(x,t)\in D_T$.
  \begin{enumerate}
    \item If $c\equiv0$ and $-\frac{\partial u}{\partial t}+Lu\le0$ both hold in $D_T$, then $$\inf_{\overline{E_T}}u=\inf_{\overline{E_T}\setminus D_T}u.$$
    \item If $c\le0$ and $-\frac{\partial u}{\partial t}+Lu\le0$ both hold in $D_T$, then
        $$\inf_{\overline{E_T}}u\ge-\sup_{\overline{E_T}\setminus D_T}u^-.$$
    \item If $c\le0$ and $-\frac{\partial u}{\partial t}+Lu=0$ both hold in $D_T$, then $$\sup_{\overline{E_T}}|u|=\sup_{\overline{E_T}\setminus D_T}|u|.$$
  \end{enumerate}
  Here the supremum and infimum may be infinity.
\end{Corollary}

\begin{Corollary}
 Let $D$ be an open and bounded set but not necessarily connected, and $E$ be an open set satisfying $D\subset E$.  Assume that $u,v\in C^{2,1}(D_T)\cap C(\overline{E_T})$, $c\le0$ in $D_T$, and $\text{\upshape supp}\,\nu(t,$ $x,\cdot)\subset\overline E-x$ for each $(x,t)\in D_T$. There is no sign condition on $c$.
  \begin{enumerate}
    \item (Comparison Principle) If $-\frac{\partial u}{\partial t}+Lu\le-\frac{\partial v}{\partial t}+Lv$ in $D_T$ and $u\ge v$ on $\overline{E_T}\setminus D_T$, then $u\ge v$ in $\overline{E_T}$.
    \item (Uniqueness) If $-\frac{\partial u}{\partial t}+Lu=-\frac{\partial v}{\partial t}+Lv$ in $D_T$ and $u=v$ on $\overline{E_T}\setminus D_T$, then $u=v$ in $\overline{E_T}$.
  \end{enumerate}
\end{Corollary}

\begin{proof}
  In the case that $c\le0$ in $D_T$, the two conclusions are trivially followed by applying Corollary \ref{p<=} to $u-v$.

  For general case without any assumption on the sign of $c$, we only need to prove that if $-\frac{\partial u}{\partial t}+Lu\le0$ in $D_T$ and $u\ge0$ on $\overline{E_T}\setminus D_T$, then $u\ge0$ in $\overline{E_T}$. Define $u^\beta:=ue^{-\beta t}$. Then $u\ge0$ is equivalent to $u^\beta\ge0$. We calculate
  \begin{equation*}
    -\frac{\partial u}{\partial t}+Lu=e^{\beta t}\Big(-\frac{\partial u^\beta}{\partial t}+Lu^\beta-\beta u^\beta\Big).
  \end{equation*}
  Hence $-\frac{\partial u}{\partial t}+Lu\le0$ is equivalent to $-\frac{\partial u^\beta}{\partial t}+(L-\beta)u^\beta\le0$. Choose a sufficiently large $\beta$, we can   ensure  $c-\beta$, the zeroth-order coefficient of operator $L-\beta$, to be nonpositive in $D_T$. By the preceding statements, we know $u^\beta\ge0$ in $\overline{E_T}$, equivalently, $u\ge0$ in $\overline{E_T}$.
\end{proof}

\begin{Corollary}
  Let $D$ be an open and bounded set but not necessarily connected, and $E$ be an open set satisfying $D\subset E$. Assume that $u,v\in C^{2,1}(D_T)\cap C(\overline{E_T})$, $-\frac{\partial u}{\partial t}+Lu=0$ in $D_T$, and $\text{\upshape supp}\,\nu(t,x,\cdot)\subset\overline E-x$ for each $(x,t)\in D_T$. If $\max_{\overline{D_T}} c\le\beta<0$, then
  \begin{equation*}
      \max_{\overline{E_T}}|u|\le e^{\beta T}\max_{\overline{E_T}\setminus D_T}|u|.
    \end{equation*}
\end{Corollary}

\begin{proof}
  We consider the function $u^\beta:=ue^{-\beta t}$ like in the previous corollary. By the same argument we know that $-\frac{\partial u}{\partial t}+Lu=0$ is equivalent to $-\frac{\partial u^\beta}{\partial t}+(L-\beta)u^\beta=0$.  The zeroth-order coefficient of operator $L-\beta$, i.e., $c-\beta$,  is nonpositive in $D_T$. Therefore, Assertion 3 of Corollary \ref{p<=} implies that
  \begin{equation*}
    e^{-\beta T}\max_{\overline{E_T}}|u|\le\max_{\overline{E_T}}|u^\beta|= \max_{\overline{E_T}\setminus D_T}|u^\beta|\le\max_{\overline{E_T}\setminus D_T}|u|.
  \end{equation*}
  Our result follows.
\end{proof}

\subsection{Strong maximum principle for parabolic case}

We now turn to the strong maximum principle for the parabolic Waldenfels operator $-\frac{\partial}{\partial t}+L$.
\begin{Theorem}[Strong maximum principle for parabolic Waldenfels operators]\label{SMPp}
Let $D$ be an open and connected set but not necessarily bounded,  and $E$ be an open set  satisfying $D\subset E$.  Assume that $u\in C^{2,2}(D_T)\cap C(\overline{E_T})$, $-\frac{\partial u}{\partial t}+Lu\ge0$ in $D_T$, and $\text{\upshape supp}\,\nu(t,x,\cdot)\subset\overline E-x$ for each $(x,t)\in D_T$. Moreover, assume that the mapping $(x,t)\to\nu(t,x,\cdot)$ is continuous in $D_T$. If one of the following conditions holds:
  \begin{enumerate}
    \item $c\equiv0$ in $D_T$ and $u$ achieves a (finite) maximum over $\overline{E_T}$ at a point $(x^0,t^0)\in D_T$;
    \item $c\le0$ in $D_T$ and $u$ achieves a (finite) nonnegative maximum over $\overline{E_T}$ at a point $(x^0,t^0)\in D_T$;
    \item $u$ achieves a zero maximum over $\overline{E_T}$ at a point $(x^0,t^0)\in D_T$,
  \end{enumerate}
  then $u$ is constant on $\overline{D_{t^0}}$, where $D_{t^0}=D\times(0,t^0]$.
\end{Theorem}

  A result of strong maximum principle for viscosity solutions of certain \emph{nonlinear} nonlocal parabolic operators proved in \cite{Cio} required a ``nondegeneracy'' condition, which is crucial in that context. But our strong maximum principle for \emph{linear} nonlocal parabolic operator in Theorem \ref{SMPp} does not need this or any other conditions like this.

The converse case that $-\frac{\partial u}{\partial t}+Lu\le0$ in $D_T$ is immediate.

\begin{Corollary}\label{sp_<}
Let $D$ be an open and connected set but not necessarily bounded,  and $E$ be an open set satisfying $D\subset E$.   Assume that $u\in C^{2,2}(D_T)\cap C(\overline{E_T})$, $-\frac{\partial u}{\partial t}+Lu\le0$ in $D_T$, and $\text{\upshape supp}\,\nu(t,x,\cdot)\subset\overline E-x$ for each $(x,t)\in D_T$. Moreover, assume that the mapping $(x,t)\to\nu(t,x,\cdot)$ is continuous in $D_T$. If one of the following conditions holds:
  \begin{enumerate}
    \item $c\equiv0$ in $D_T$ and $u$ achieves a (finite) minimum over $\overline{E_T}$ at a point $(x^0,t^0)\in D_T$;
    \item $c\le0$ in $D_T$ and $u$ achieves a (finite) nonpositive minimum over $\overline{E_T}$ at a point $(x^0,t^0)\in D_T$;
    \item $u$ achieves a zero minimum over $\overline{E_T}$ at a point $(x^0,t^0)\in D_T$,
  \end{enumerate}
  then $u$ is constant on $\overline{D_{t^0}}$, where $D_{t^0}=D\times(0,t^0]$.
\end{Corollary}

To prove the strong maximum principle, we  will consider the horizontal propagation of maximizer  point in space by the similar arguments in elliptic case,   and further obtain the vertical propagation of maximizer point locally in time by the weak maximum principle  in elliptic case.

Denote $M:=\max_{\overline{E_T}}u<\infty$ for convenience. Under the assumptions in Theorem \ref{SMPp}, that is, $u\in C^{2,2}(D_T)\cap C(\overline{E_T})$, $-\frac{\partial u}{\partial t}+Lu\ge0$ in $D_T$, and $u(x^0,t^0)=M$ with point $(x^0,t^0)\in D_T$,   $\text{\upshape supp}\,\nu(t,x,\cdot)\subset\overline E-x$ for each $(x,t)\in D_T$, and the mapping $(x,t)\to\nu(t,x,\cdot)$ is continuous in $D_T$.  Furthermore, one of the following   assumptions holds:
\begin{Assumption}\label{c1}
  $c\equiv0$ in $D_T$.
\end{Assumption}
\begin{Assumption}\label{c2}
  $c\le0$ in $D_T$ and $M\ge0$.
\end{Assumption}
\begin{Assumption}\label{c3}
  $M=0$.
\end{Assumption}
The following lemma follows from Theorem \ref{smpe} and Remark \ref{psd}.
\begin{Lemma}\label{step1}
  Let $B\subset\mathbb{R}^{d+1}$ be a open ball with $\overline{B}\subset D_T$. Assume that  there exists a point $(x^0,t^0)\in\partial B$ such that $u(x^0,t^0)=M$ and $u(x,t)<M$ for each point $(x,t)\in B$. Then $t^0$ is either the smallest or the largest value over all the time coordinates of points in $\overline B$.
\end{Lemma}

\begin{proof}
  If $t^0=T$, the theorem is trivial. Hence we assume $t^0<T$.  Equivalently, $(x^0,t^0)$ is an interior point of $D_T$.

  We regard the parabolic Waldenfels operator $-\frac{\partial}{\partial t}+L$ as a degenerate elliptic Waldenfels operator by writing
  \begin{equation*}
    \begin{split}
         & \Big(-\frac{\partial}{\partial t}+L\Big)u(x,t) \\
        =& -\frac{\partial u}{\partial t}(x,t)+\text{tr}[\mathbf{a}^T(\nabla^2 u)](x,t)+b^T\nabla u(x,t)+cu(x,t) \\
         & +\int_{\R^d\setminus\{0\}}\big[u(x+z,t)-u(x,t)-z^T\nabla u(x,t)\cdot\mathbf{1}_{\{|z|<1\}}\big]\nu(t,x,dz) \\
         =& \ \text{tr}\bigg[\begin{pmatrix}\mathbf{a} & 0 \\0 & 0\end{pmatrix}^T(\nabla^2_{x,t}u)\bigg](x,t)+ (b^T,-1)\cdot\nabla_{x,t}u(x,t)+cu(x,t) \\
         & +\int_{\R^{d+1}}\big[u(x+z,t+s)-u(x,t) \\
         &\qquad\qquad -(z^T,s)\cdot \nabla_{x,t}u(x,t)\cdot\mathbf{1}_{\{|z|^2+|s|^2\leq 1\}}\big]\nu(t,x,dz)\delta_0(ds).
    \end{split}
  \end{equation*}
  Thus, we can replace the matrix $\mathbf{a}$ in the elliptic Waldenfels operator (\ref{ewo1}) by $\hat{\mathbf{a}}=\begin{pmatrix}\mathbf{a} & 0 \\0 & 0\end{pmatrix}$, vector $b$ by $\hat{b}=\begin{pmatrix}b\\-1\end{pmatrix}$, and  the kernel $\nu(x,dz)$ by $\hat{\nu}((x,t),dzds):=\nu(t,x,dz)\delta_0(ds)$.

  Now we verify that the kernel $\hat{\nu}$, defined on $\mathbb{R}^{d+1}$, satisfies the moment condition (\ref{LM}), although its support is not contained in  $\mathbb{R}^{d+1}\setminus\{\mathbf{0}\}$. By recalling (\ref{LMt}), condition (\ref{LM}) for $\hat{\nu}$ immediately follows as we see that
  \begin{equation}\label{lm}
    \begin{split}
        & \int_{\mathbb{R}^{d+1}}\big[1\land(|z|^2+|s|^2)\big]\hat{\nu} (x,t,dzds) \\
        =& \int_{\mathbb{R}^{d+1}}\big[1\land(|z|^2+|s|^2)\big] \nu(t,x,dz)\delta_0(ds) \\
         =& \int_{\mathbb{R}^d\setminus\{0\}} (1\land|z|^2)\nu(t,x,dz) \\
         <& \ \infty.
    \end{split}
  \end{equation}
  Since   $(x^0,t^0)$ is the maximizer  point over $\overline{E_T}$ of $u$, we may replace the set $E_0$ in (\ref{E0}) by
  \begin{equation*}
    \hat{E_0}:=\{(x,t)\in\overline{E}\times\{t^0\}\mid u(x,t)=M\}.
  \end{equation*}
   Due to the fact that the support of measure $\hat{\nu}((x^0,t^0),\cdot)$ is contained in $\overline{E}\times\{0\}$,  we  can further  derive (\ref{nu}) for $\hat{\nu}$ and $\hat{E_0}$. It turns out that (\ref{Kv}) and (\ref{I1}), (\ref{I2}) also hold in this situation.

  Combining Remark \ref{psd} and the   preceding arguments , we conclude that, as in Theorem \ref{smpe}, $\frac{\partial u}{\partial\mathbf{n}}>0$ holds in the case that the unit outer normal vector $\mathbf{n}$ of $(x^0,t^0)$ over $\partial B$ is not in the nullspace of $\mathbf{a}(x^0,t^0)$, which is exactly the space $N:=\{(x,t)\mid x=0\}$. But this leads to a contradiction: since $u$ attains a maximum at the interior point $(x^0,t^0)$, we have $\frac{\partial u}{\partial\mathbf{n}}=0$.
  Therefore, $(x^0,t^0)$ must be a pole of ball $B$, whose unit outer normal vector is just in $N$.
  This completes the proof.
\end{proof}

The next lemma shows that for every $t\in(0,T)$, we have either $u(x,t)<M$ or $u(x,t)=M$ for all $x\in D$.  This  means that the non-maximizer point (or   the maximizer point) may propagation horizontally in space. The proof can be found in \cite{Renardy}, we do not present it here and the main points of the proof can be found in Section \ref{appendix}.

\begin{Lemma}\label{lem1}
  Assume that $u(x^0,t^0)<M$ with $x^0\in D$ and $t^0\in(0,T)$. Then $u(x,t^0)<M$ for every $x\in D$.
\end{Lemma}

\begin{Remark}
  In Lemma \ref{lem1}, we restrict the set $D$ to be connected to make sure that the point $(x^1,t^0)$ can be chosen. More generally, if $D$ is not connected, we may replace $D$ in the previous proof with the connected component of $D$ which contains the maximizer  point. We thus    conclude  that for every fixed $t\in(0,T)$,    either $u(x,t)<M$ or $u(x,t)=M$ holds in each connected component of $D$. Then as in Remark \ref{component}, we see that the diffusion term gives rise to the horizontal propagation of maximizer point in the corresponding connected component.

  As in Remark \ref{tran},  or from \cite{Cio,Cov}, we also see that  the horizontal propagation of maximizer point by translation of measure support. Namely, if $u(x^0,t^0)=M$ with $x^0\in D$ and $t^0\in(0,T)$, then $u\equiv M$ on the set $\overline{\bigcup_{n=0}^\infty\varLambda_n}$, where $\tilde\varLambda_n$'s are defined by induction,
  \begin{equation*}
    \tilde\varLambda_0 = {x^0}, \tilde\varLambda_{n+1}= \bigcup_{x\in D\cap\tilde\varLambda_n}[\text{\upshape supp}\,\nu(t^0,x,\cdot)+x].
  \end{equation*}
  Thus in this scheme, the jump diffusion term leads to the horizontal propagation of maximizer point between those connected components, since jumps from one connected component to another might occur when measure supports overlap two or more connected components.
\end{Remark}

Furthermore, we present the final lemma we need. It means the maximizer point may propagate vertically in time in a local sense. The proof can also be found in Section \ref{appendix}
\begin{Lemma}\label{lem2}
  Assume that $u<M$ in $D\times(t^0,t^1)$, with  $0\le t^0<t^1\le T$. Then $u<M$ in $D\times\{t^1\}$.
\end{Lemma}

\vspace{3mm}
Finally we can prove Theorem \ref{SMPp}.

\begin{proof}[Proof of Theorem \ref{SMPp}]
Set $D_T^<:=\{(x,t)\in D_T\mid u(x,t)<M\}$. Then $D_T^<$ is a relatively open subset of $D_T$. From Lemma \ref{lem1}, we know that for each fixed $t\in(0,T)$,  either $u(x,t)<M$ or $u(x,t)=M$  holds for all $x\in D$. Therefore, $D_T^<$   must be   of the form $D\times I$,  for some $I\subset (0,T]$ relatively open in $(0,T]$.

For fixed $s\in I$ and $s\neq T$, define $\tau(s):=\sup\{t\mid (s,t)\subset I\}$, where the set in supremum is never empty as $I$ is relatively open in $(0,T]$. From Lemma \ref{lem2}, we see $\tau(s)\in I$. Then $(r,\tau(s)]$ is the connected component in $I$ containing $s$, for some $r<s$. Consequently $\tau(s)=T$, since   $I$ is relatively open. Thus we summarize that for each $s\in I$, $[s,T]\subset I$, which is trivial when $s=T$. Hence, the relatively open set $I$ only has  two options: either $[0,T]$ or $(s,T]$ for some $s\in [0,T)$. In light of the fact $u(x^0,t^0)=M$, or equivalently $t^0\in I$, we conclude that $I$ must be of the form $(s,T]$ for some $s\in[t^0,T)$, as required.  This finishes the proof of   Theorem \ref{SMPp}.
\end{proof}

\section{Examples}\label{exp}

We will give some examples in this section. These examples are all concerned with symmetric $\alpha$-stable L\'evy noise which are not covered in Taira's framework in \cite{KTaira}, since the jump measure is of unbounded support.

\begin{Example}[Mean exit time] \label{exit}
Consider a stochastic system in $\R^d$:
$$dX_t = b(X_t) dt + dW_t + dL_t^\alpha,$$
where $W_t$ is a standard Wiener process, and $L_t^\alpha$ is a L\'evy process with jump measure $\nu_\alpha(dz) = c_{\alpha,d} \frac{dz}{|z|^{d+\alpha}}$, for   $\alpha \in (0, 2)$ and $c_{\alpha,d}$ a positive constant depending on $\alpha$ and $d$, together with zero drift and zero diffusion. The generator for this system is the following elliptic Waldenfels operator
\begin{equation*}
  \begin{split}
     Lu(x) = & \frac12\Delta u(x)+b(x)\cdot\nabla u(x) \\
       & +\int_{\R^d\setminus\{0\}}\Big[u(x+z)-u(x)-\sum^d_{j=1}z_j \frac{\partial u}{\partial x_j}(x)\mathbf{1}_{\{|z|<1\}}\Big]\nu_\alpha(dz).
  \end{split}
\end{equation*}

Let $D$ be a domain in $\R^d$. The mean exit time for  $X_t$, starting at $x$, exits firstly from $D$ is denoted by $\tau(x)$. By   Dynkin formula for such jump diffusion process \cite{App,Kolok}, as shown in \cite{Dua15, Oks,Sa}, we know that $\tau$ satisfies the following equation,
\begin{equation*}
  \begin{cases}
    L \tau = -1 &\text{in }D, \\
    \tau = 0 &\text{in }D^c.
  \end{cases}
\end{equation*}
By the strong maximum principle \ref{smpe}, or precisely Corollary \ref{s_Lu<} with the special case $E=\R^d$, we conclude that the mean exit time $\tau$ cannot take zero value inside $D$,  unless it is constant (inside the domain $D$).
\end{Example}

\begin{Example}[Escape probability] \label{escape}
Similarly, let $U$ be a subset of $D^c$. The likelihood that $X_t$, starting at $x$, exits firstly from $D$ by landing in the target set $U$ is called the escape probability from $D$ to $U$, denoted by $p(x)$. As shown in \cite{Kan, Liao}, the escape probability $p$ satisfies the following equation,
\begin{equation*}
  \begin{cases}
    L p = 0 &\text{in }D, \\
    p = 1 &\text{in }U, \\
    p = 0 &\text{in }D^c\setminus U.
  \end{cases}
\end{equation*}
By Theorem \ref{smpe} and Corollary \ref{s_Lu<} with $E=\R^d$, we conclude that $p$ cannot take values of zero or one at any point inside $D$.
\end{Example}

\begin{Example}[Fokker-Planck equation] \label{FPE2013}
Consider a stochastic system in $\R^d$:
\begin{equation}\label{rand}
  \left\{
  \begin{aligned}
     dX_t & = b(X_t) dt + dW_t + dL_t^\alpha, \\
       X_0 & =x^0,
  \end{aligned}
  \right.
\end{equation}
where  $W_t$ is a standard Wiener process and $L_t^\alpha$ is a L\'evy process with pure jump measure $\nu_\alpha(dz) = c_{\alpha,d}\frac{dz}{|z|^{d+\alpha}}$, for   $\alpha \in (0, 2)$ and $c_{\alpha,d}$ a positive constant depending on $\alpha$ and $d$.
The Fokker-Planck equation for the probability density of the solution, as   shown in \cite{Dua15,Garroni},  is
\begin{equation}\label{FPE2014}
  \begin{split}
     \frac{\partial p}{\partial t} =&\ \frac12 \Delta p -  b(x) \cdot  \nabla p - (\nabla \cdot b) p \\
       & +\int_{\R^d\setminus\{0\}}\Big[p(x+z,t)-p(x,t)-\sum^d_{j=1}z_j\frac{\partial
p}{\partial x_j}(x,t)\mathbf{1}_{\{|z|<1\}}\Big]\nu_\alpha(dz).
  \end{split}
\end{equation}

Let $D$ be a domain in $\R^d$. In this case, the coefficient of zeroth-order term is $c=-\nabla\cdot b$. We apply the strong maximum principle in Theorem \ref{SMPp} and Corollary \ref{sp_<} with $E=\R^d$. If $\nabla\cdot b\equiv0$, which means the deterministic vector field of stochastic system (\ref{rand}) is divergence-free,
then the probability density $p$ cannot attain its maximum (or minimum) over $\R^d\times[0,\infty)$ in $D\times[0,\infty)$, unless it is constant at all time before the maximizer (or minimizer) point. Moreover,
if $\nabla\cdot b\ge0$,
then $p$ cannot attain its   maximum or zero minimum over $\R^d\times[0,\infty)$ in $D\times[0,\infty)$ (note that $p$ only takes nonnegative values), unless it is constant at all time before this point as well.
\end{Example}

\section{Appendix: Proofs for lemmas}\label{appendix}

The proofs of some technical lemmas will be presented here for the sake of completeness.

\subsection{Proof of Lemma \ref{lem1}}

  Assume on the contrary that  $D\times\{t^0\}$ contains some points at which $u=M$. Choose the point $(x^1,t^0)$ nearest to $(x^0,t^0)$ at the line $\{(x,t^0)\mid x\in D\}$, such that $u(x^1,t^0)=M$, which is possible since $D$ is connected and the set $\{(x,t)\in D_T\mid u(x,t)=M\}$ (or simply, $\{u=M\}$) is relatively closed in $D_T$. Denote by $l$ the line segment connecting $(x^0,t^0)$ with $(x^1,t^0)$, and set
  \begin{equation*}
    \delta=\min\{|x^1-x^0|,\text{dist}(l,\partial D)\}.
  \end{equation*}
  For $x\in l_0:=\{x\mid(x,t^0)\in l,0<|x-x^1|<\delta\}$, define
  \begin{equation*}
    \rho(x)=\text{dist}((x,t^0),(D_T)^\circ\cap\{u=M\}).
  \end{equation*}
  Obviously, $0<\rho(x)\le|x-x^1|$ in $l_0$. See Figure \ref{fig:Lemma1}.

\begin{figure}[!bth]
  \centering
  \includegraphics[width=.6\textwidth]{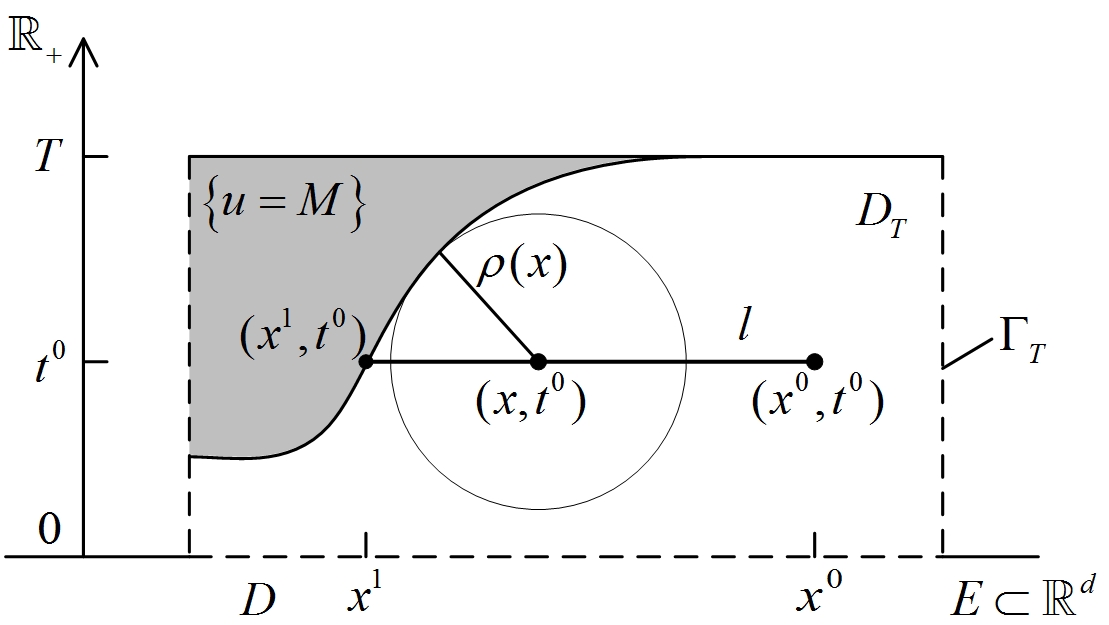}
  \caption{Sketch for Lemma \ref{lem1}.} \label{fig:Lemma1}
\end{figure}

  Consider next the open ball $B:=B((x,t^0),\rho(x))\subset\R^{d+1}$ with center $(x,t^0)$ and radius $\rho(x)$. Then $\overline{B}\subset D_T$, $u<M$ in $B$ and $\partial B$ contains points where $u=M$. Thus by Lemma \ref{step1} we conclude that either $u(x,t^0+\rho(x))=M$ or $u(x,t^0-\rho(x))=M$. By the Pythagorean theorem, we assert
  \begin{equation*}
    \rho(x+\epsilon e)^2\le\rho(x)^2+\epsilon^2,
  \end{equation*}
  for sufficiently small $|\epsilon|>0$, where $e$ is the unit vector along $l$.
  In the same way, $\rho(x)^2\le\rho(x+\epsilon e)^2+\epsilon^2$.
  We note that
  \begin{equation*}
    \lim_{\epsilon\to0} \frac{\sqrt{\rho(x)^2-\epsilon^2}-\rho(x)}{\epsilon} \le\lim_{\epsilon\to0}\frac{\rho(x+\epsilon e)-\rho(x)}{\epsilon}\le\lim_{\epsilon\to0} \frac{\sqrt{\rho(x)^2+\epsilon^2}-\rho(x)}{\epsilon},
  \end{equation*}
  which implies that
  \begin{equation*}
    \frac{d\sqrt{\rho(x)^2-\epsilon^2}}{d\epsilon}\Bigg|_{\epsilon=0} \le\frac{d\rho(x+\epsilon e)}{d\epsilon}\Bigg|_{\epsilon=0} \le\frac{d\sqrt{\rho(x)^2+\epsilon^2}}{d\epsilon}\Bigg|_{\epsilon=0}.
  \end{equation*}
This leads to
  \begin{equation*}
    \frac{d\rho(x+\epsilon e)}{d\epsilon}\Bigg|_{\epsilon=0}=0,
  \end{equation*}
  for each $x\in l_0$. Consequently, $\rho(x)$ is a constant in $l_0$, which is a contradiction since $\rho(x)>0$ in $l_0$ and $\rho(x)\to0$ as $x\to x^1$. The proof is complete.

\subsection{Proof of Lemma \ref{lem2}}

  Assume that there exists a point $x^1\in D$ such that $u(x^1,t^1)=M$. Let $\hat{B}:=B((x^1,t^1),r)\subset\R^{d+1}$ be such a small ball that $\hat{B}\cap(E\times(t^0,t^1))$ is contained in $D_T$. Define
  \begin{equation*}
    v(x,t):=e^{-|x-x^1|^2-\beta(t-t^1)}-1 \quad\text{for }(x,t)\in \overline{E}\times[t^0,t^1],
  \end{equation*}
  where $\beta$ is a positive constant as selected below. Then for $(x,t)\in \hat{B}\cap(E\times(t^0,t^1))$,
  \begin{equation*}
    \frac{\partial v}{\partial t} =\ -\beta e^{-|x-x^1|^2-\beta(t-t^1)},
  \end{equation*}
  and also
  \begin{align*}
    Av =&\ e^{-|x-x^1|^2-\beta(t-t^1)}\Big\{\text{tr}\big[\mathbf{a}^T \big(4(x-x^1)\otimes(x-x^1)-2I\big)\big]-2b^T(x-x^1) \\
       &\qquad\qquad\qquad\quad\ +c\big(1-e^{-\beta(t^1-t)+|x-x^1|^2}\big)\Big\} \\
       =&\ e^{-|x-x^1|^2-\beta(t-t^1)}\big[4(x-x^1)^T\mathbf{a}^T(x-x^1) -2\text{tr}(\mathbf{a})-2b^T(x-x^1) \\
       &\qquad\qquad\qquad\quad\ +c\big(1-e^{-\beta(t^1-t)+|x-x^1|^2}\big)\big] \\
       \ge&\ e^{-|x-x^1|^2-\beta(t-t^1)}\big[4\gamma|x-x^1|^2 -2\text{tr}(\mathbf{a})-2|b||x-x^1|-|c|\big],
  \end{align*}
  for  $\beta>0$ large enough. We set
  \begin{equation*}
    \big|4\gamma|x-x^1|^2 -2\text{tr}(\mathbf{a})-2|b||x-x^1|-|c|\big|\le C_1<\infty,
  \end{equation*}
  for some positive constant $C_1$,   independent of $\beta$.
  Moreover,  we have
  \begin{equation*}
    \begin{split}
       Kv(x,t)= &\ e^{-|x-x^1|^2-\beta(t-t^1)}\int_{\R^d\setminus\{0\}} \big[e^{-|x+z-x^1|^2+|x-x^1|^2}-1 \\
         & \qquad\qquad\qquad\qquad\qquad\ +2z^T(x-x^1)\mathbf{1}_{\{|z|<1\}}\big]\nu(t,x,dz).
    \end{split}
  \end{equation*}
 Using Taylor expansions and the fact that each $\nu(t,x,\cdot)$ is L\'evy measure, we know that
  \begin{equation*}
    \bigg|\int_{\R^d\setminus\{0\}} \big[e^{-|x+z-x^1|^2+|x-x^1|^2}-1+2z^T(x-x^1)\mathbf{1}_{\{|z|<1\}}\big]\nu(t,x,dz)\bigg|\le C_2<\infty,
  \end{equation*}
  where $C_2$ is a positive constant independent of $\beta$. Hence, in $\hat{B}\cap(E\times(t^0,t^1))$
  \begin{equation}\label{Ltv}
    \begin{split}
       -\frac{\partial v}{\partial t}+Lv & \ge-\frac{\partial v}{\partial t}+Av-|Kv| \\
         & \ge e^{-|x-x^1|^2-\beta(t-t^1)}(\beta-C_1-C_2) \\
         & \ge 0,
    \end{split}
  \end{equation}
  provided $\beta$ is sufficiently large.

  \begin{figure}[!bth]
    \centering
    \includegraphics[width=.6\textwidth]{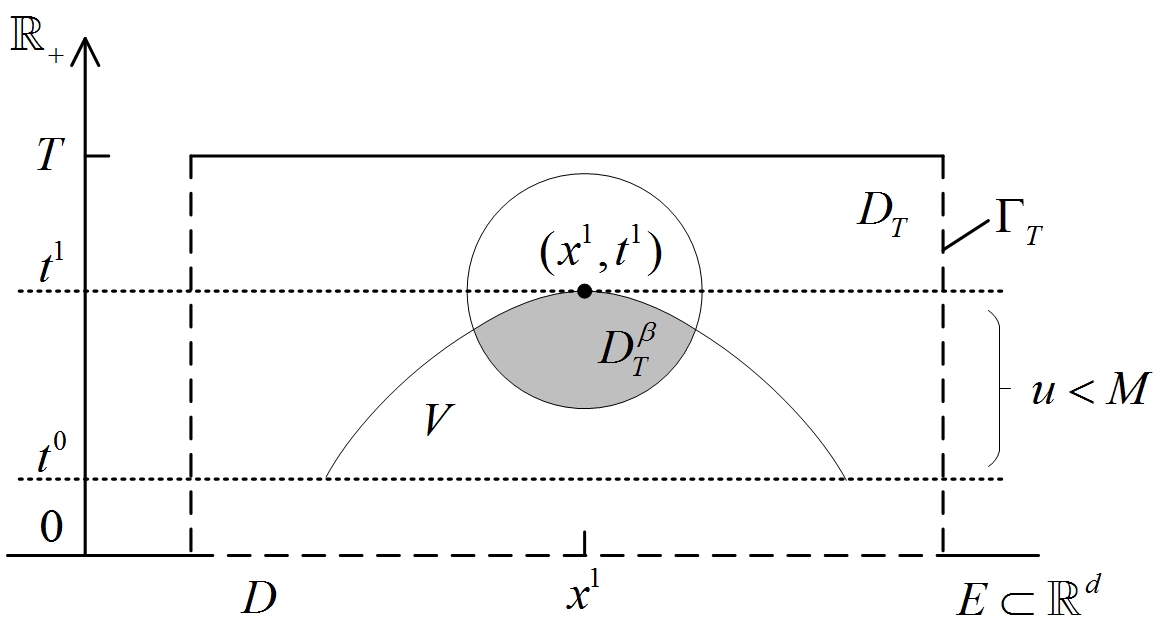}
    \caption{Sketch for Lemma \ref{lem2}.} \label{fig:Lemma2}
  \end{figure}

  Next limit our attention within the domain
  \begin{equation*}
    D_T^\beta:=\{(x,t)\in \hat{B}\mid|x-x^1|^2+\beta(t-t^1)<0\},
  \end{equation*}
  and define
  \begin{equation*}
    u^\epsilon=u+\epsilon v-M \quad\text{on }\overline{E}\times[t^0,t^1],
  \end{equation*}
  for a constant $\epsilon>0$. Then choosing $\epsilon$ small enough, we have
  \begin{equation*}
    u^\epsilon\le0 \quad \text{on }(\overline{E}\times[t^0,t^1])\setminus D_T^\beta,
  \end{equation*}
  since $v\le0$ on $\overline E\times[t^0,t^1]\setminus V$, and $u<M$ on $V\setminus D_T^\beta$, where $V:=\{(x,t)\in E\times(t^0,t^1)\mid|x-x^1|^2+\beta(t-t^1)<0\}= (E\times(t^0,t^1))\cap\{v>0\}$. See Figure \ref{fig:Lemma2}.

  Under Assumptions \ref{c1} and \ref{c2}, $c\equiv0$ in $D_T$, or $c\le0$ in $D_T$ also $M\ge0$, from (\ref{Ltv}) and the assumption $-\frac{\partial u}{\partial t}+Lu\ge0$ in $D_T$, we compute
  \begin{equation*}
    -\frac{\partial u^\epsilon}{\partial t}+Lu^\epsilon\ge-cM\ge0 \quad\text{in }D_T^\beta.
  \end{equation*}
  Then Remark \ref{natural}, or more precisely, the weak maximum principle for elliptic case in Theorem \ref{WMPe}, implies that $u^\epsilon\le0$ throughout $\overline E\times[t^0,t^1]$. But $u^\epsilon(x^1,t^1)=0$, and thus at the point $(x^1,t^1)$,
  \begin{equation*}
    0\le\frac{\partial u^\epsilon}{\partial t}=\frac{\partial u}{\partial t}+\epsilon\frac{\partial v}{\partial t}=\frac{\partial u}{\partial t}-\epsilon\beta,
  \end{equation*}
  that is, $\frac{\partial u}{\partial t}(x^1,t^1)>0$. However, as explained already in Theorem \ref{WMPe} or Theorem \ref{WMPp}, we know that $Lu\le0$ at point $(x^1,t^1)$. Hence
  \begin{equation*}
    -\frac{\partial u}{\partial t}+Lu<0 \quad\text{ at } (x^1,t^1).
  \end{equation*}
  This is a contradiction.

  Under Assumption \ref{c3}, we may use the same argument as in the end of the proof of Theorem \ref{smpe}, that is, replacing $L$ and $c$ respectively with $L-c^+$ and $c-c^+$ and applying the acquired result for Assumption \ref{c2}, the same result for this case follows.

\medskip

\textbf{Acknowledgement}. This work was partly supported by the National Science Foundation grant
DMS-1620449. We appreciate the helpful comments and suggestions of the anonymous reviewer.

\end{document}